\documentclass[11pt,reqno,a4paper]{amsart}

\usepackage{comment}

\usepackage{graphicx}

\usepackage{amsmath,amsfonts,amsthm,mathrsfs,amssymb,cite}
\usepackage[usenames]{color}

\newtheorem{thm}{Theorem}[section]

\newtheorem{lem}{Lemma}[section]
\newtheorem{prop}{Proposition}[section]

\theoremstyle{definition}
\newtheorem{defn}{Definition}[section]
\theoremstyle{remark}

\newtheorem{rem}{Remark}[section]
\numberwithin{equation}{section}

\title[Decoupling Elastic Waves and Its Applications]{Decoupling Elastic Waves and Its Applications}

\author{Hongyu Liu}
\address{Department of Mathematics, Hong Kong Baptist University, Kowloon, Hong Kong SAR.}
\email{hongyuliu@hkbu.edu.hk}

\author{Jingni Xiao}
\address{Department of Mathematics, Hong Kong Baptist University, Kowloon, Hong Kong SAR.}
\email{xiaojn@live.com}

\begin{document}

\begin{abstract}

In this paper, we consider time-harmonic elastic wave scattering governed by the Lam\'e system. It is known that the elastic wave field can be decomposed into the shear and compressional parts, namely, the pressure and shear waves that are generally coexisting, but propagating at different speeds. We consider the third or fourth kind scatterer and derive two geometric conditions, respectively, related to the mean and Gaussian curvatures of the boundary surface of the scatterer that can ensure the decoupling of the shear and pressure waves. Then we apply the decoupling results to the uniqueness and stability analysis for inverse elastic scattering problems in determining polyhedral scatterers by a minimal number of far-field measurements.

\medskip

 \smallskip


\end{abstract}

\maketitle

\tableofcontents

\section{Introduction}\label{sect:1}

In this paper, we are mainly concerned with the time-harmonic elastic wave scattering due to an incident plane wave and an impenetrable scatterer. Let $D\subset\mathbb{R}^3$ be a bounded domain with a Lipschitz boundary $\partial D$ and a connected complement $\mathbb{R}^3\backslash\overline{D}$. $D$ represents an impenetrable elastic scatterer embedded in an infinite isotropic and homogenous elastic medium in $\mathbb{R}^3$ that is characterized by the Lam\'e constants $\lambda$ and $\mu$ satisfying $\mu>0$ and $3\lambda+2\mu>0$. Consider a time-harmonic elastic plane wave $U^{in}(x)$, $x=(x_j)_{j=1}^3\in\mathbb{R}^3$ (with the time variation of the form $e^{-i\omega t}$ being factorized out, where $\omega\in\mathbb{R}_+$ denotes the angular wavenumber) impinging on the scatterer $D$. The propagation of the elastic plane wave will be interrupted and this leads to the so-called scattering. The elastic scattering is governed by the Lam\'e system (or the reduced Navier equation) (cf. \cite{Kup, Landau}),
\begin{equation}\label{eq:lames}
(\Delta^*+\omega^2) U=0\quad\mbox{in}\quad\mathbb{R}^3\backslash\overline{D},
\end{equation}
where $U(x)=(U_j)_{j=1}^3\in\mathbb{C}^3$ denotes the total displacement field. In \eqref{eq:lames}, the operator $\Delta^*$ is defined by
\begin{equation}\label{eq:deltas}
\begin{split}
[\Delta^*U]_j=&\sum_{l=1}^3\frac{\partial}{\partial x_l}\left\{\mu\left(\frac{\partial U_j}{\partial x_l}+\frac{\partial U_l}{\partial x_j} \right)\right\}+\frac{\partial}{\partial x_j}\left[\lambda\nabla\cdot U\right]\\
=&\mu\Delta U_j+(\lambda+\mu)\frac{\partial}{\partial x_j}[\nabla\cdot U]\\
=&-\mu[\nabla\wedge\nabla\wedge U]_j+(\lambda+2\mu)\frac{\partial}{\partial x_j}[\nabla\cdot U],\ \ j=1,2,3.
\end{split}
\end{equation}

Throughout, the incident/impinging elastic plane wave is of the following general form
\begin{equation}\label{eq:planewave}
\begin{split}
U^{in}(x)&=U^{in}(x; d, d^{\perp},\alpha_p,\alpha_s,\omega)\alpha_p d e^{ik_pd\cdot x}+\alpha_s d^{\perp} e^{ik_sd\cdot x}\\
&:=U^{in}_p+U^{in}_s,\qquad d, d^{\perp}\in\mathbb{S}^2,\quad \alpha_p, \alpha_s\in\mathbb{C},
\end{split}
\end{equation}
where $d\in\mathbb{S}^2:=\{x\in\mathbb{R}^3; \|x\|=1\}$, is the incident/impinging direction; $d^\perp\in\mathbb{S}^2$ satisfying $d\cdot d^\perp=0$ is the polarization direction; and $k_s:=\omega/\sqrt{\mu}$, $k_p:=\omega/\sqrt{\lambda+2\mu}$ respectively denote the shear and compressional wavenumbers.
{In what follows, for a given vector $d\in\mathbb{S}^2$, $d^\perp$ shall always denote a unit vector that is perpendicular to $d$.}
If $\alpha_p\neq 0$ while $\alpha_s=0$ for $U^{in}$ in \eqref{eq:planewave}, then $U^{in}=U^{in}_p:=\alpha_p d e^{ik_p x\cdot d}$ is the plane pressure wave; whereas if $\alpha_p=0$ and $\alpha_s\neq 0$, then $U^{in}=U^{in}_s:=\alpha_s d^\perp e^{ik_p x\cdot d}$ is the plane shear wave. The total elastic wave field $U$ in \eqref{eq:lames} is composed of the incident plane wave $U^{in}$ in \eqref{eq:planewave} and the scattered wave $U^{sc}$,
\begin{equation}\label{eq:scatteredwave}
U(x)=U^{in}(x)+U^{sc}(x),\quad x\in\mathbb{R}^3\backslash\overline{D}.
\end{equation}
It is straightforward to verify that the scattered wave field $U^{sc}=U-U^{in}$ satisfy the Lam\'e system
\begin{equation}\label{eq:lames2}
(\Delta^*+\omega^2) U^{sc}=0\quad\mbox{in}\quad\mathbb{R}^3\backslash\overline{D}.
\end{equation}
The scattered wave $U^{sc}$ is radiating and required to satisfy the Kupradze radiation condition,
\begin{equation}\label{eq:radiation}
\begin{split}
(\nabla\wedge\nabla\wedge U^{sc})(x)\wedge\frac{x}{\|x\|}-ik_s\nabla\wedge U^{sc}(x)&=\mathcal{O}(\|x\|^{-2}),\\
\frac{x}{\|x\|}\cdot[\nabla(\nabla\cdot U^{sc})](x)-ik_p \nabla\cdot U^{sc}(x)&=\mathcal{O}(\|x\|^{-2}),
\end{split}
\end{equation}
as $\|x\|\rightarrow+\infty$, which holds uniformly in all directions $\hat{x}:=x/\|x\|$ for $x\in\mathbb{R}^3\backslash\{0\}$.

To complete the description of the elastic scattering problem, we next complement the system \eqref{eq:lames}--\eqref{eq:radiation} with a suitable boundary condition on the exterior boundary of the impenetrable scatterer $D$. Depending on the physical property of the scatterer, the following four kinds of boundary conditions have been widely considered in the literature. The first kind boundary condition is given by
\begin{equation}\label{eq:firstkind}
U=0\quad\mbox{on}\quad\partial D.
\end{equation}
The second kind boundary condition is given by
\begin{equation}\label{eq:secondkind}
TU=0\quad\mbox{on}\quad\partial D,
\end{equation}
where $T$ is the stress (or traction) operator on $\partial D$ defined by
\begin{equation}\label{eq:tractionoperator}
TU:=2\mu\partial_\nu U+\lambda\,\nu\nabla\cdot U+\mu\,\nu\wedge(\nabla\wedge U),
\end{equation}
with $\nu\in\mathbb{S}^2$ the outward unit normal vector to $\partial D$, and $\partial _\nu$ the boundary differential operator defined as
\[
\partial_\nu U:=[\nu\cdot\nabla U_{j}]_{j=1}^3.
\]
The boundary conditions of the third and fourth kinds are, respectively, given by
\begin{equation}\label{eq:thirdkind}
\nu\cdot U=0\quad\mbox{and}\quad\nu\wedge TU=0\qquad\mbox{on}\ \ \partial D,
\end{equation}
and
\begin{equation}\label{eq:fourthkind}
\nu\wedge U=0\quad\mbox{and}\quad\nu\cdot TU=0\qquad\mbox{on}\ \ \partial D.
\end{equation}

The elastic system \eqref{eq:lames}--\eqref{eq:radiation} associated with either of the four kinds of boundary conditions in \eqref{eq:firstkind}--\eqref{eq:fourthkind} is well understood with a unique solution $U\in H_{loc}^1(\mathbb{R}^3\backslash\overline{D})$. In particular, if $\partial D$ is $C^2$-smooth, then $U\in C^2(\mathbb{R}^3\backslash\overline{D})^3\cap C^1(\mathbb{R}^3\backslash D)^3$. We refer to \cite{ElY10,Kup} for the related results.

It is widely known that similar to \eqref{eq:planewave}, one has the following decomposition
\begin{equation}\label{eq:decomp1}
U^{sc}=U_p^{sc}+U_s^{sc},\ \ U_p^{sc}:=-\frac{1}{k_p^2}\nabla(\nabla\cdot U^{sc}),\ U_s^{sc}:=\frac{1}{k_s^2}\nabla\wedge (\nabla\wedge U^{sc}),
\end{equation}
where the vector functions $U_p^{sc}$ and $U_s^{sc}$ are referred to as the pressure (longitudinal) and shear (transversal) parts of $U^{sc}$, respectively, satisfying
\begin{align}
(\Delta+k_p^2) U_p^{sc}=0,\quad \nabla\wedge U_p^{sc}=0\quad &\mbox{in}\ \ \mathbb{R}^3\backslash\overline{D},\label{eq:decompp}\\
(\Delta+k_s^2) U_s^{sc}=0,\quad \nabla\cdot U_s^{sc}=0\quad &\mbox{in}\ \ \mathbb{R}^3\backslash\overline{D},\label{eq:decomps}.
\end{align}
Moreover, the Kupradze radiation condition \eqref{eq:radiation} is equivalent to the Sommerfeld radiation condition,
\begin{eqnarray}
\frac{x}{\|x\|}\cdot[\nabla U_p^{sc}]_j-ik_p[U_p^{sc}]_j & = & \mathcal{O}(\|x\|^{-2}),\label{eq:radiations1}\\
\frac{x}{\|x\|}\cdot[\nabla U_s^{sc}]_j-ik_s[U_s^{sc}]_j & = & \mathcal{O}(\|x\|^{-2}), \label{eq:radiations2}
\end{eqnarray}
uniformly in all directions $\hat x$ as $\|x\|\rightarrow+\infty$, for $j=1,2,3$. In terms of \eqref{eq:planewave} and \eqref{eq:decomp1}, one has the corresponding decomposition of the total wave field as follows,
\begin{equation}\label{eq:decompt}
U=U_p+U_s,\quad U_p:=U_p^{in}+U_p^{sc},\ \ U_s:=U_s^{in}+U_p^{sc}.
\end{equation}

Using the decomposition described above, there holds the following asymptotic expansion as $\|x\|\rightarrow+\infty$,
\begin{equation}\label{eq:farfield}
U^{sc}(x)=\frac{\exp(ik_p\|x\|)}{\|x\|} U_p^\infty(\hat x)+\frac{\exp(ik_s\|x\|)}{\|x\|} U_s^\infty(\hat x)+\mathcal{O}\left(\frac{1}{\|x\|^2}\right).
\end{equation}
$U_p^{\infty}$ and $U_s^\infty$ are known as the P-part (longitudinal part) and S-part (transversal part) of the far-field pattern of $U^{sc}$. Moreover, using the decomposition \eqref{eq:decomp1}, $U_p^\infty$ and $U_s^\infty$ are the far-field patterns of $U_p^{sc}$ and $U_s^{sc}$, respectively. In this paper, we define the total field pattern to be
\begin{equation}\label{eq:farfieldt}
U_t^\infty(\hat x):=U_p^\infty(\hat x)+U_s^\infty(\hat x).
\end{equation}
It is easy to verify that
\begin{equation}\label{eq:farfieldd}
U_p^\infty(\hat x)=[U_t^\infty(\hat x)\cdot \hat x]\hat x\quad\mbox{and}\quad U_s^\infty(\hat x)=[\hat x\wedge U_t^\infty(\hat x)]\wedge\hat x.
\end{equation}

Though the elastic wave field $U^{sc}$ can be decomposed into its shear and compressional parts, the two types of waves are generally coupled together. In fact, by using the boundary conditions \eqref{eq:firstkind}--\eqref{eq:fourthkind}, one can readily see that the coupling occurs on the boundary of the scatterer $D$. That means, in general, the shear wave and the pressure wave coexist, even in the case that the incident wave is a pure plane pressure wave or a pure plane shear wave. In this article, we shall derive two geometric conditions in terms of the mean and Gaussian curvatures of the boundary surface of the scatterer $D$, that can ensure the decoupling of the pressure and shear waves for the case with the third or fourth kind boundary condition. Clearly, if the pressure and shear waves are decoupled, then the study of the Lam\'e system \eqref{eq:lames} can be significantly simplified to the investigation of the two vectorial Helmholtz systems \eqref{eq:decompp} and \eqref{eq:decomps}. This undoubtedly would find important application in relevant study of the elastic wave scattering. Indeed, another main motivation for our this study is the inverse elastic scattering problem which concerns the determination of the scatterer $D$ by knowledge of the corresponding far-field pattern $U_\tau^\infty(\hat x)$, where $\tau=t, p$ or $s$. If one introduces an operator $\mathcal{F}$ which sends the unknown/inaccessible scatterer $D$ to the corresponding far-field pattern $U_\tau^\infty$ associated with the incident/impinging plane wave $U^{in}$ in \eqref{eq:planewave}, then the inverse scattering problem can be abstractly formulated as follows,
\begin{equation}\label{eq:ip}
\mathcal{F}(D)=U_\tau^\infty(\hat x; d, d^\perp,\alpha_p,\alpha_s,\omega, D),\ \ \hat x\in\mathbb{S}^2;\ \ \tau=t, p\ \mbox{or}\ s,
\end{equation}
where we write $U_\tau^\infty(\hat x; d, d^\perp,\alpha_p,\alpha_s,\omega, D)$ to specify its dependence on the incident plane wave as well as the underlying scatterer. The inverse problem \eqref{eq:ip} is of significant practical importance and has been extensively investigated in the literature. It is easily seen that the inverse problem \eqref{eq:ip} is nonlinear and it is known to be ill-posed in the sense of Hadamard; we refer to \cite{Amm1,Amm2,ElY10,Hu1,Hu2,NG1,NG2} and the references therein for related results on the inverse problem.

In this paper, we are particularly interested in the case that a single far-field measurement is used for the inverse problem \eqref{eq:ip}. That is, the far-field pattern $U_\tau^\infty$ is collected corresponding to a single incident plane wave with $d$, $d^\perp$, $\alpha_p$, $\alpha_s$ and $\omega$ being fixed. The inverse problem is formally-determined with a single far-field measurement. The inverse elastic scattering problem with a single far-field measurement still remains to be a challenging open problem in the literature. To our best knowledge, the only existing general uniqueness result in the literature for the inverse problem \eqref{eq:ip} with a single far-field pattern was derived in a recent work \cite{ElY10} for the determination of polyhedral scatterers endowed with the third or fourth kind boundary condition. The argument therein is based a reflection principle for the Navier equation associated with the third or fourth kind boundary condition, and the path argument developed in \cite{Liu1,Liu2} for inverse acoustic and electromagnetic scattering governed by the Helmholtz and Maxwell equations, respectively. In the current study, the polyhedral scatterer is a particular case where the pressure and shear waves can be decoupled. Using the decoupling result, we are able to derive much more general uniqueness results which cover the uniqueness results in \cite{ElY10} as two particular cases. Moreover, by using the decoupling result, we can advance the relevant study much further in deriving optimal stability estimates through the extension of the stability results in our recent work \cite{LMRX}. These are the first stability results in the literature in determining elastic obstacles by a single far-field measurement. 

The rest of the paper is structured as follows. In Section~\ref{sect:2}, we derive the geometric conditions for the decoupling of elastic waves. Section 3 is devoted to a case study when the underlying scatterer is a ball. In Section~4, we apply the decoupling result to the study of the uniqueness and stability for the inverse elastic scattering problems. The paper is concluded in Section 5 with some relevant discussion.

\section{Decoupling elastic waves}\label{sect:2}

In this section, we consider the decoupling of the pressure and shear waves for the elastic system \eqref{eq:lames}--\eqref{eq:radiation} associated with, respectively, the third kind boundary condition \eqref{eq:thirdkind} and the fourth kind boundary condition \eqref{eq:fourthkind}. For notational convenience, we write $U(x; U^{in}, [D, \mathrm{III}])=U(x; d,d^\perp,$ $\alpha_p,\alpha_s,\omega,[D, \mathrm{III}])$ and $U(x; U^{in}, [D, \mathrm{IV}])=U(x; d,d^\perp,\alpha_p,\alpha_s,\omega,[D, \mathrm{IV}])$ to denote the total wave fields corresponding the third and fourth kinds scatterers, respectively. The same notations shall be employed for the scattered wave fields and the far-field patterns, as well as the scattering problems associated with the first and second kinds of boundary conditions. Throughout the rest of the paper, we assume that the Lipschitz scatterer $D$ has a piecewise $C^{2,1}$-smooth boundary. Since the Lam\'e system \eqref{eq:lames} is strongly elliptic, we know that the $H^{1}_{loc}(\mathbb{R}^3\backslash\overline{D})$ solution is locally $H^3$ near a regular piece of $\partial D$ (see \cite{McL}).

Let $\Gamma$ be a regular piece of $\partial D$ and be represented in the following parametric form,
\begin{equation}\label{eq:pf1}
x(u)=\big(x_1(u_1,u_2), x_2(u_1,u_2), x_3(u_1,u_2)\big)^{\mathrm{T}},\ \ u=(u_1,u_2)^{\mathrm{T}}\in\mathbb{R}^2,
\end{equation}
where the superscript $\mathrm{T}$ denotes the transpose of a matrix. Denote by $g=(g_{jk})_{j,k=1}^2$ the first fundamental matrix of differential geometry for $\Gamma$ with
\begin{equation}\label{eq:fdmm}
g_{jk}:=\frac{\partial x}{\partial u_j}\cdot\frac{\partial x}{\partial u_k},\ \ \ j, k=1,2.
\end{equation}
Henceforth, we assume that the parametric form \eqref{eq:pf1} is such chosen that
\begin{equation}\label{eq:pf2}
\nu=\nu(x)=\frac{1}{\sqrt{|g|}}\frac{\partial x}{\partial u_1}\wedge\frac{\partial x}{\partial u_2},\ \ \ |g|:=\mbox{det}(g),
\end{equation}
is the outward unit normal vector of $\partial D$ on $\Gamma$. In what follows, we let
\begin{equation}\label{eq:curvatures}
H(x)=\frac 1 2 (\kappa_1+\kappa_2)(x)\quad\mbox{and}\quad K(x)=(\kappa_1\kappa_2)(x), \quad x\in \Gamma,
\end{equation}
denote, respectively, the mean and Gaussian curvatures at $x\in \Gamma$, where $\kappa_1$ and $\kappa_2$ are the principle curvatures of $\Gamma$ at $x\in\Gamma$. We also denote by $\mathrm{Grad}_\Gamma$ the surface gradient operator on $\Gamma$ (cf. \cite{CK12,Ned}), and recall that for any $\varphi\in C^1(\partial D)$,
\begin{equation}\label{eq:sg1}
\mathrm{Grad}_\Gamma\varphi=\sum_{j,k=1}^2 g^{jk}\frac{\partial\varphi}{\partial u_j}\frac{\partial x}{\partial u_k},
\end{equation}
where
\begin{equation}\label{eq:sg2}
(g^{jk})_{j,k=1}^2:=\big[(g_{jk})_{j,k=1}^2\big ]^{-1}.
\end{equation}

\subsection{The fourth kind boundary condition}\label{sect:2.1}

We first treat the scattering problem associated with the fourth kind boundary condition and have that

\begin{thm}\label{thm:fourthkind}
Let $\Gamma\subset\partial D$ be a regular piece of the boundary surface of the scatterer $D$. Suppose that there holds
\begin{equation}\label{eq:gcmc}
H(x)= 0\quad\mbox{for}\ \ x\in\Gamma.
\end{equation}
Then for $U=(U_j)_{j=1}^3$ a solution to the Lam\'e system \eqref{eq:lames}--\eqref{eq:radiation} satisfying the fourth kind boundary condition on $\Gamma$,
\begin{equation}\label{eq:fourthkindga}
\nu\wedge U=0\qquad\mbox{and}\qquad \nu\cdot TU=0\qquad\mbox{on}\ \ \Gamma,
\end{equation}
if and only if %the decomposed pressure and shear waves $U_p$ and $U_s$ defined in \eqref{eq:decompt} satisfy
\begin{equation}\label{eq:decompps2}
\nabla\cdot U=0\qquad\mbox{and}\qquad \nu\wedge\left(\nabla\wedge (\nabla\wedge U)\right)=0\qquad\mbox{on}\ \ \Gamma.
\end{equation}
\end{thm}

\begin{proof}

Let $\nu(x)=(\nu_j(x))_{j=1}^3\in\mathbb{S}^2$ be the unit normal vector to $x\in\Gamma$ pointing to the exterior of $D$.  Define
\begin{equation}\label{eq:matrix1}
\mathcal{D}[\nu](x):=\Big[\mathrm{Grad}_\Gamma\nu_1(x), \mathrm{Grad}_\Gamma\nu_2(x), \mathrm{Grad}_\Gamma\nu_3(x)\Big],\quad x\in\Gamma.
\end{equation}
Since $2H(x)=-\nabla\cdot \nu(x)$ (cf. \cite{Ned}), \eqref{eq:gcmc} actually implies that
\begin{equation}\label{eq:gc1}
\mathrm{tr}(\mathcal{D}[\nu](x)):=\sum_{j=1}^3\Big[\mathrm{Grad}_\Gamma\nu_j(x)\Big]_j\equiv 0\quad\mbox{for}\ \ x\in\Gamma.
\end{equation}

Next, we show that \eqref{eq:fourthkindga} implies \eqref{eq:decompps2}. By direct calculations, it is easy to show that
\begin{equation}\label{eq:p1}
\begin{split}
&\mathrm{Grad}_\Gamma\Big(\nu_l U_j\Big)=\nu_l\mathrm{Grad}_\Gamma U_j+U_j\mathrm{Grad}_\Gamma\nu_l\\
=& \nu_l\nabla U_j-\nu_l\Big(\nu\cdot\nabla U_j\Big)\nu+ U_j\mathrm{Grad}_\Gamma\nu_{l},\quad j,l=1,2,3.
\end{split}
\end{equation}
Noting that $\nu\cdot\mathrm{Grad}_\Gamma\nu_l=0$, $l=1,2,3$, then by using the fact that $\nu\wedge U=0$ on $\Gamma$ in \eqref{eq:fourthkindga}, we have
\begin{equation}\label{eq:p2}
U\cdot\mathrm{Grad}_\Gamma \nu_l=0\qquad\mbox{on}\ \ \Gamma,\quad l=1,2,3.
\end{equation}
and
\begin{equation}\label{eq:p3}
\mathrm{Grad}_\Gamma\Big( \big[ \nu\wedge U\big]_j\Big)=0\quad \mbox{on}\ \ \Gamma, \ \ j=1,2,3.
\end{equation}
Hence, by virtue of \eqref{eq:p1}, together with straightforward computations, one further has that
\begin{equation}\label{eq:p4}
\begin{split}
0=&\sum_{l=1}^3\Big[\mathrm{Grad}_\Gamma\big(\nu_l U_j-\nu_j U_l \big) \Big]_l\\
=& \nu_j\big(\nu\cdot\partial_\nu U-\nabla\cdot U \big)+U_j\sum_{l=1}^3\Big[\mathrm{Grad}_\Gamma\nu_l\Big]_l \ \mbox{on}\ \Gamma, j=1,2,3.
\end{split}
\end{equation}
Multiplying both sides of \eqref{eq:p4} by $\nu_j$ and then summarizing the resulting equalities for $j=1,2,3$, one then arrives at
\begin{equation}\label{eq:p5}
\nu\cdot\partial_\nu U=\nabla\cdot U-(\nu\cdot U)\sum_{j=1}^3\Big[\mathrm{Grad}_\Gamma \nu_j\Big]_j\quad\mbox{on}\ \ \Gamma.
\end{equation}
By applying \eqref{eq:gc1} to \eqref{eq:p5}, we clearly have $\nu\cdot \partial_\nu U=\nabla\cdot U$ on $\Gamma$ and as a consequence, by using $\nu\cdot TU=0$ in \eqref{eq:fourthkindga}, there holds
\begin{equation}\label{eq:p6}
\begin{split}
0=& \nu\cdot TU=2\mu\nu\cdot\partial_\nu U+\lambda\nabla\cdot U\\
=& (2\mu+\lambda)\nabla\cdot U\quad \mbox{on}\ \ \Gamma.
\end{split}
\end{equation}
Finally, it is easy to have that
\begin{equation}\label{eq:p8}
\frac{1}{k_s^2}\nu\wedge\left(\nabla\wedge (\nabla\wedge U)\right)=\nu\wedge U+\frac{1}{k_p^2}\nu\wedge \Big(\nabla\big(\nabla\cdot U\big)\Big)=0\quad\mbox{on}\ \ \Gamma,
\end{equation}
which together with \eqref{eq:p6} clearly yields \eqref{eq:decompps2}.

We proceed to show that \eqref{eq:decompps2} implies \eqref{eq:fourthkindga}. In view of the two conditions in \eqref{eq:decompps2}, one readily has
\begin{equation}\label{eq:p10}
\nu\wedge U=\frac{1}{k_s^2}\nu\wedge\left(\nabla\wedge (\nabla\wedge U)\right)-\frac{1}{k_p^2}\nu\wedge\big(\nabla(\nabla\cdot U)\big)=0\ \ \mbox{on}\ \Gamma.
\end{equation}
Furthermore, by using \eqref{eq:p5} and \eqref{eq:gc1}, one clearly has $\nu\cdot\partial_\nu U=\nabla\cdot U=0$ on $\Gamma$, which in turn implies that
\begin{equation}
\begin{split}
\nu\cdot TU=& 2\mu\nu\cdot\partial_\nu U+\lambda\nabla\cdot U\\
=& (2\mu+\lambda)\nabla\cdot U=0\quad\mbox{on}\ \ \Gamma.
\end{split}
\end{equation}

The proof is complete.
\end{proof}

\begin{rem}\label{rem:decomp4th}
It is easily seen from its proof that Theorem~\ref{thm:fourthkind} holds in a more general scenario where $U$ is not necessarily a solution to the scattering system \eqref{eq:lames}--\eqref{eq:radiation}. Indeed, Theorem~\ref{thm:fourthkind} still holds, as long as $U$ is an $H^1$-solution to the Lam\'e system \eqref{eq:lames} in a neighborhood of $\Gamma$. However, in order to have a consistent exposition, we stick to the scattering system \eqref{eq:lames}--\eqref{eq:radiation} for the decoupling results, and the same remark applies to our subsequent decoupling study in the present section.
\end{rem}

\begin{rem}
From a geometric point of view, the condition \eqref{eq:gcmc} means that if $\Gamma$ is a \emph{minimal surface}, then the two boundary conditions \eqref{eq:fourthkindga} and \eqref{eq:decompps2} are equivalent to each other. 
\end{rem}

The following proposition is a direct consequence of Theorem~\ref{thm:fourthkind}.

\begin{prop}\label{prop:decom}
Let $D=[D,\mathrm{IV}]$ be a fourth kind elastic scatterer with a piecewise $C^{2,1}$-smooth boundary such that each of its regular piece satisfies the geometric condition \eqref{eq:gcmc}. Let $U([D,\mathrm{IV}])=U(\cdot; U^{in}, [D, \mathrm{IV}])$ with the incident field $U^{in}$ given by \eqref{eq:planewave}. 

Set
\begin{equation}\label{eq:dsp21}
v_p=v_p([D,\mathrm{IV}]):=-\nabla\cdot U([D,\mathrm{IV}]),\quad v_p^{sc}:=-\nabla\cdot U^{sc}([D,\mathrm{IV}]).
\end{equation}
Then $v_p$ is a solution to the following problem,
\begin{equation}\label{eq:dsp3}
\begin{cases}
& (\Delta+k_p^2) v_p=0\qquad\mbox{in}\ \ \mathbb{R}^3\backslash\overline{D},\medskip\\
& v_p=v_p^{in}+v_p^{sc}\qquad \mbox{in}\ \ \mathbb{R}^3\backslash\overline{D},\medskip\\
& v_p=0\qquad\mbox{on}\ \ \partial D,\medskip\\
& \displaystyle{\frac{x}{\|x\|}\cdot(\nabla v_p^{sc})-ik_p v_p^{sc}=\mathcal{O}(\|x\|^{-2})}\ \mbox{as}\ \|x\|\rightarrow+\infty,
\end{cases}
\end{equation}
with
\begin{equation}\label{eq:dsp22}
v_p^{in}:=-\nabla\cdot U_p^{in}=-(i\alpha_pk_p)e^{ik_px\cdot d}.
\end{equation}

In parallel, if we set
\begin{equation}\label{eq:dsp5half11}
\begin{split}
& E_s=E_s([D,\mathrm{IV}]):=\nabla\wedge U([D,\mathrm{IV}]),\  E_s^{sc}:=\nabla\wedge U^{sc}([D,\mathrm{IV}]),\\
& H_s=H_s([D,\mathrm{IV}]):=\nabla\wedge E_s/(ik_s), 
\end{split}
\end{equation}
then $(E_s,H_s)$ solves the following problem,
\begin{equation}\label{eq:dsp61}
\begin{cases}
& \nabla\wedge E_s-ik_s H_s=0\qquad \mbox{in}\ \ \mathbb{R}^3\backslash\overline{D},\medskip\\
& \nabla\wedge H_s+ik_s E_s=0\qquad \mbox{in}\ \ \mathbb{R}^3\backslash\overline{D},\medskip\\
& E_s=E_s^{in}+E_s^{sc}\qquad \mbox{in}\ \ \mathbb{R}^3\backslash\overline{D},\medskip\\
& \nu\wedge \left( \nabla\wedge E_s\right) =0\qquad \mbox{on}\ \ \partial D,\medskip\\
& \displaystyle{\frac{x}{\|x\|}\wedge (\nabla\wedge E^{sc}_s) + ik_sE^{sc}_s =\mathcal{O}(\|x\|^{-2})}\ \mbox{as}\ \|x\|\rightarrow+\infty, %E_s^{sc}\ \mbox{satisfies {\color{red} the radiation condition \eqref{eq:radiations}}},
\end{cases}
\end{equation}
with
\begin{equation}\label{eq:dsp5half12}
E_s^{in}:=\nabla\wedge U^{in}=i\alpha_s k_s\, (d\wedge d^\perp)e^{ik_sd\cdot x}.
\end{equation}
\end{prop}
\begin{proof}
By using the decompositions in \eqref{eq:decomp1} and \eqref{eq:decompt}, it is straightforward to show that
\begin{equation}\label{eq:aaa1}
-\nabla\cdot U=-\nabla\cdot U_p\quad\mbox{and}\quad \nabla\wedge U=\nabla\wedge U_s\quad\mbox{in}\ \ \mathbb{R}^3\backslash\overline{D}.
\end{equation}
Hence, by noting \eqref{eq:decompp} and \eqref{eq:decomps}, both $v_p$ in \eqref{eq:dsp21} and $E_s$ in \eqref{eq:dsp5half11} satisfy the vectorial Helmholtz equation. It is also noted that the Maxwell equations in \eqref{eq:dsp61} for $(E_s, H_s)$ is equivalent to that $E_s$ satisfies the vectorial Helmholtz equation (cf. \cite{CK12}). The boundary conditions in \eqref{eq:dsp21} and \eqref{eq:dsp61}, respectively, for $v_p$ and $E_s$, are immediate consequences of Theorem~\ref{thm:fourthkind}. Finally, the radiation conditions can be verified by direct calculations. 

The proof is complete. 
\end{proof}

\eqref{eq:dsp3} is known as the Helmholtz system the describes the time-harmonic acoustic scattering problem, where $D$ denotes a \emph{sound-soft} scatterer and, $v_p^{in}$, $v_p^{sc}$ and $v_p$ are respectively called the incident wave, scattered wave and total wave fields. \eqref{eq:dsp61} is known as the Maxwell system that describes the time-harmonic electromagnetic scattering problem, where $D$ represents a \emph{perfect magnetic conducting} (PMC) scatterer, and $E_s^{in}$, $E_s^{sc}$, $E_s$ and $H_s$ are respectively the incident electric wave, the scattered electric wave, the total electric wave and the total magnetic wave fields (see, e.g., \cite{CK12}).
In what follows, we shall write $D=[D,\mathrm{SS}]$ to specify that $D$ is sound-soft scatterer corresponding to the Helmholtz system \eqref{eq:dsp3}; whereas $D=[D,\mathrm{PMC}]$ means that $D$ is a PMC scatterer associated with the EM scattering problem \eqref{eq:dsp61}.

It is immediately observed from the definitions \eqref{eq:dsp21} and \eqref{eq:decompt} (see also \eqref{eq:aaa1}) that 
\begin{equation}\label{eq:relap}
U_p(U^{in}; [D,\mathrm{IV}])=\frac{1}{k_p^2}\nabla v_p([D,\mathrm{IV}]),
\end{equation}
and similarly,
\begin{equation}\label{eq:relas}
U_s(U^{in}; [D,\mathrm{IV}])=\frac{1}{k_s^2}\nabla\wedge E_s([D,\mathrm{IV}]).
\end{equation}
In what follows, we shall write
\begin{equation}\label{eq:ind11}
\begin{split}
-\nabla\cdot U(U^{in}; [D,\mathrm{IV}])=& v_p(v_p^{in}; [D,\mathrm{SS}]),\\
U_p(U^{in}; [D,\mathrm{IV}])=& \frac{1}{k_p^2}\nabla v_p(v_p^{in}; [D,\mathrm{SS}]),
\end{split}
\end{equation}
and
\begin{equation}\label{eq:ind2f1}
\begin{split}
\nabla\wedge U(U^{in}; [D,\mathrm{IV}])=& E_s(E_s^{in}; [D,\mathrm{PMC}]),\\
U_s(U^{in}; [D,\mathrm{IV}])=& \frac{1}{k_s^2}\nabla\wedge E_s(E_s^{in}; [D,\mathrm{PMC}]),
\end{split}
\end{equation}
to indicate the correspondences between the solution to the elastic scattering problem and those to the acoustic and the electromagnetic scattering problems that are stated in Proposition~\ref{prop:decom}.

Due to the symmetry role of the electric and magnetic fields in the Maxwell system \eqref{eq:dsp61}, as an alternative to \eqref{eq:dsp5half11}, one can also set
\begin{equation}\label{eq:dsp51}
\begin{split}
& \widetilde{E}_s=\widetilde{E}_s([D,\mathrm{IV}]):=H_s([D,\mathrm{IV}]),\ \widetilde{E}_s^{sc}([D,\mathrm{IV}]):=\widetilde{E}_s-\widetilde{E}_s^{in},\\
& \widetilde{H}_s=\widetilde{H}_s([D,\mathrm{IV}]):=-\frac{i}{k_s}\nabla\wedge \widetilde{E}_s([D,\mathrm{IV}]),
\end{split}
\end{equation}
with
\begin{equation}\label{eq:dsp52}
\widetilde{E}_s^{in}:=-\frac{i}{k_s}\nabla\wedge\left( \nabla\wedge  U^{in}\right) =-i\alpha_s k_s\, d^\perp e^{ik_sd\cdot x}.
\end{equation}
This would lead to a similar Maxwell system to \eqref{eq:dsp61} for $(\widetilde{E}_s,\widetilde{H}_s)$, but with the homogeneous boundary condition on $\partial D$ replaced by
\begin{equation}\label{eq:dsp5half21}
\nu\wedge \widetilde{E}_s=0\ \ \mbox{on}\ \ \partial D.
\end{equation}
In this case, $[D,\mathrm{IV}]$ is referred to as a {\it perfect electric conducting} (PEC) obstacle, and we shall write $[D,\mathrm{PEC}]$ when it is employed in the EM problem \eqref{eq:dsp61} with the boundary condition replaced by \eqref{eq:dsp5half21}. Furthermore, in such a case, \eqref{eq:ind2f1} should be modified to be
\begin{equation}\label{eq:ind21}
%\frac{1}{k_s^2}\nabla\wedge \left( \nabla\wedge U(U^{in}; [D,\mathrm{IV}])\right)\equiv
U_s(U^{in}; [D,\mathrm{IV}]) =\frac{i} {k_s} E_s(\widetilde{E}_s^{in}; [D,\mathrm{PEC}]).
\end{equation}

Clearly, according to Proposition~\ref{prop:decom}, if a fourth kind scatterer $D$ satisfies the geometric condition \eqref{eq:gcmc}, then the corresponding elastic scattering system can be decoupled into the two scattering systems \eqref{eq:dsp3} and \eqref{eq:dsp61}, which involves only the pressure and shear waves, respectively, through the relations \eqref{eq:relap} and \eqref{eq:relas}; see Theorem~\ref{thm:3} in Section~\ref{sect:4} for a summarizing conclusion. This fact would be of significant importance when considering the inverse problem of determining the scatterer $D$ from the corresponding elastic far-field measurements. Indeed, it can be seen that the information of the scatterer is also implicitly contained in the two simpler systems \eqref{eq:dsp3} and \eqref{eq:dsp61}. We shall make extensive use of such a decoupling idea for our study in Section~\ref{sect:4} for the inverse elastic scattering problems.

Next, we show that the geometric condition \eqref{eq:gc1} on the scatterer $[D,\mathrm{IV}]$ is also ``generically" necessary for the decoupling of the pressure and shear waves. Indeed, we have that

\begin{thm}\label{thm:fourthkind2}
Let $D$ be a fourth kind scatterer such that
\begin{equation}\label{eq:gcmc2}
H(x)\neq 0\quad\mbox{for a.e. }\ x\in\partial D.
\end{equation}
Let $U$ be a solution to the Lam\'e system \eqref{eq:lames}--\eqref{eq:radiation}. Suppose further that
\begin{equation}\label{eq:gc1n2}
U(\cdot; U^{in}, [D, \mathrm{IV}])\equiv\hspace*{-3.5mm}\backslash\ U(\cdot; U^{in}, [D, \mathrm{I}])\quad\mbox{in}\ \ \mathbb{R}^3\backslash\overline{D}.
\end{equation}
Then \eqref{eq:fourthkindga} and \eqref{eq:decompps2} cannot hold true simultaneously on $\partial D$.
\end{thm}

\begin{proof}

We prove the theorem by an absurdity argument. Assume contrarily that both \eqref{eq:fourthkindga} and \eqref{eq:decompps2} hold on $\partial D$ for the scattering system \eqref{eq:lames}--\eqref{eq:radiation}. First, by using \eqref{eq:tractionoperator}, one has by a direct calculation that
\begin{equation}\label{eq:pp1}
\nu\cdot TU=2\mu\nu\cdot\partial_\nu U+\lambda\nabla\cdot U.
\end{equation}
Then, by using \eqref{eq:decompps2} on $\partial D$ and the fact that $\nabla\cdot U_s\equiv 0$, one has
\begin{equation}\label{eq:pp2}
\nabla\cdot U=\nabla\cdot(U_p+U_s)=\nabla\cdot U_p=0\quad\mbox{on}\ \ \partial D.
\end{equation}
Inserting \eqref{eq:pp2} into \eqref{eq:pp1}, and using \eqref{eq:fourthkindga} on $\partial D$ that $\nu\cdot TU=0$, one readily has
\begin{equation}\label{eq:pp3}
\nu\cdot\partial_\nu U=0\quad\mbox{on}\ \ \partial D.
\end{equation}

Next, by using \eqref{eq:fourthkindga} on $\partial D$ that $\nu\wedge U=0$, and a completely similar argument to the first part of the proof of Theorem~\ref{thm:fourthkind}, one can show that \eqref{eq:p5} also holds on $\partial D$, namely,
\begin{equation}\label{eq:pp4}
\nu\cdot\partial_\nu U=\nabla\cdot U-(\nu\cdot U)\sum_{j=1}^3\Big[\mathrm{Grad}_\Gamma \nu_j\Big]_j\quad\mbox{on}\ \ \partial D.
\end{equation}
Clearly, \eqref{eq:gcmc2} implies that
\begin{equation}\label{eq:gc1n1}
\mathrm{tr}(\mathcal{D}[\nu](x))\neq 0\quad\mbox{for a.e.}\ x\in\partial D,
\end{equation}
where $\mathrm{tr}(\mathcal{D}[\nu])$ is defined in \eqref{eq:gc1}.
Combining \eqref{eq:gc1n1}, \eqref{eq:pp2} and \eqref{eq:pp3} and \eqref{eq:pp4}, one can readily verify that
\begin{equation}\label{eq:pp5}
\nu\cdot U=0\quad\mbox{on}\ \ \partial D,
\end{equation}
which together with the fact again that $\nu\wedge U=0$ one $\partial D$ immediately implies that
\[
U=0\quad\mbox{on}\ \ \partial D.
\]
That is, $D$ is also a first kind scatterer and hence $U(\cdot; U^{in}, [D, \mathrm{IV}])=U(\cdot; U^{in}, [D, \mathrm{I}])\ \mbox{in}\ \mathbb{R}^3\backslash\overline{D}$, which is a contradiction to \eqref{eq:gc1n2}.

The proof is complete.
\end{proof}

\begin{rem}\label{eq:generic1}
We would like to remark that physically speaking, \eqref{eq:gc1n2} means that two scatterers possessing the same shape $D$ but of different physical properties must produce different scattering effects due an incident plane wave $U^{in}$. Hence, the condition should be a generic one. This is also related to the inverse scattering problem of identifying the physical property of an unknown/inaccessible scatterer by using a single wave measurement due to the incident plane wave $U^{in}$.
\end{rem}

In Section~\ref{sect:3}, we shall consider the case that $D=B_R$, a central ball of radius $R\in\mathbb{R}_+$. In this case, one has that
\begin{equation}\label{eq:ball1}
H(x)\equiv -\frac{1}{R}.
\end{equation}
Hence, the condition \eqref{eq:gcmc2} holds true. By using the method of wave series expansion, we show that \eqref{eq:fourthkindga} and \eqref{eq:decompps2} cannot hold true simultaneously on $\partial D$. This reinforce our conclusion that the geometric condition \eqref{eq:gc1} is both sufficient and ``generically" necessary for the decoupling of the pressure and shear waves of the elastic wave scattering.

\subsection{The third kind boundary condition}

In this section, we treat the elastic scattering problem associated with the third kind boundary condition. Our argument shall follow a similar spirit to the fourth kind boundary condition case in Section~\ref{sect:2.1}, but would become a bit more complicated.

First, we have

\begin{thm}\label{thm:thirdkind}
Let $\Gamma\subset\partial D$ be a regular piece of the boundary of the scatterer $D$. Suppose that there hold
\begin{equation}\label{eq:gc32}
H(x)=0\quad\mbox{and}\quad K(x)=0\quad\mbox{for}\ \ x\in\Gamma.
\end{equation}
Then for $U=(U_j)_{j=1}^3$ a solution to the Lam\'e system \eqref{eq:lames}--\eqref{eq:radiation} satisfying the third kind boundary condition on $\Gamma$,
\begin{equation}\label{eq:thirdkindga}
\nu\cdot U=0\qquad\mbox{and}\qquad \nu\wedge TU=0\qquad\mbox{on}\ \ \Gamma,
\end{equation}
if and only 
\begin{equation}\label{eq:decompps232}
\nu\cdot \left( \nabla\,(\nabla \cdot U)\right) =0\qquad\mbox{and}\qquad \nu\wedge(\nabla\wedge U)=0\qquad\mbox{on}\ \ \Gamma.
\end{equation}
\end{thm}

\begin{rem}
The geometric condition \eqref{eq:gc32} actually implies that $\Gamma$ is flat, namely $\nu\equiv \mbox{constant}$ on $\Gamma$. Stating it in the form \eqref{eq:gc32} would give more insights about the geometric condition, especially when treating its``necessity" in a certain sense; see \eqref{eq:gc1n131} in Theorem~\ref{thm:thirdkind2}.
\end{rem}

\begin{proof}
We first show that \eqref{eq:thirdkindga} implies \eqref{eq:decompps232}. By using $\nu\cdot U=0$ in \eqref{eq:thirdkindga} on $\Gamma$, one has by straightforward calculations that
\begin{equation}\label{eq:pp31}
\begin{split}
0=&\nu\wedge\mathrm{Grad}_\Gamma(\nu\cdot U)\\
=&\nu\wedge\sum_{j=1}^3 \nu_j\nabla U_j+\nu\wedge\sum_{j=1}^3 U_j\mathrm{Grad}_\Gamma\nu_j\\
=&\nu\wedge\mathscr{D}_\nu U\quad\mbox{on}\quad \Gamma,
\end{split}
\end{equation}
where the differential operator $\mathscr{D}_\nu$ is defined by
\begin{equation}\label{eq:pp32}
\mathscr{D}_\nu U:=\big[\nu\cdot\partial_j U \big]_{j=1}^3\in\mathbb{C}^3.
\end{equation}
On the other hand, we have
\begin{equation}\label{eq:pp37}
\mathscr{D}_\nu U=\partial_\nu U+\nu\wedge(\nabla\wedge U).
\end{equation}
Hence, by using $\nu\wedge TU=0$\ on $\Gamma$ in \eqref{eq:thirdkindga} and \eqref{eq:pp37}, one further has
\begin{equation}\label{eq:pp38}
0=\nu\wedge TU=2\mu\nu\wedge\mathscr{D}_\nu U-\mu\nu\wedge(\nu\wedge (\nabla\wedge U))\ \ \mbox{on}\ \Gamma,
\end{equation}
which together with \eqref{eq:pp31} immediately implies that
\begin{equation}\label{eq:pp39}
\nu\wedge(\nabla\wedge U)=0\quad\mbox{on}\ \ \Gamma.
\end{equation}
Furthermore, one has
\begin{equation}\label{eq:pp41}
\nu\cdot\left( \nabla\wedge(\nabla\wedge U)\right) =\mathrm{Div}_\Gamma\big(\nu\wedge (\nabla\wedge U)\big)=0\quad\mbox{on} \ \ \Gamma,
\end{equation}
where $\mathrm{Div}_\Gamma$ denotes the surface divergence operator on $\Gamma$ (cf. \cite[Section 3]{CK12}). Thus, one has by \eqref{eq:thirdkindga} and \eqref{eq:pp41} that
\begin{equation}\label{eq:pp42}
\frac{1}{k_p^2}\nu\cdot \left( \nabla\,(\nabla \cdot U)\right)=\nu\cdot U-\frac{1}{k_s^2}\nu\cdot\nabla\wedge(\nabla\wedge U)=0\quad\mbox{on}\ \ \Gamma.
\end{equation}

Next, we show that \eqref{eq:decompps232} implies \eqref{eq:thirdkindga}. First, by using the condition $\nu\wedge (\nabla\wedge U)=0$ on $\Gamma$ in \eqref{eq:decompps232}, similar to \eqref{eq:pp41}, one readily has
\begin{equation}\label{eq:aaa2}
\nu\cdot(\nabla\wedge(\nabla\wedge U))=0\quad\mbox{on}\ \ \Gamma.
\end{equation}
Then using the first relation in \eqref{eq:pp42}, i.e.,
\begin{equation}\label{eq:aaa3}
\frac{1}{k_p^2}\nu\cdot \left( \nabla\,(\nabla \cdot U)\right)=\nu\cdot U-\frac{1}{k_s^2}\nu\cdot\nabla\wedge(\nabla\wedge U),
\end{equation}
and the fact that $\nu\cdot \left( \nabla\,(\nabla \cdot U)\right) =0$ on $\Gamma$ in \eqref{eq:decompps232}, one readily has $\nu\cdot U=0$. Hence, there holds that
\begin{equation}\label{eq:aaa4}
\nu\wedge\mathscr{D}_\nu U=\nu\wedge\mathrm{Grad}_\Gamma(\nu\cdot U)=0\quad\mbox{on}\ \ \Gamma.
\end{equation}
Finally, by combining \eqref{eq:pp37}, \eqref{eq:decompps232} and \eqref{eq:aaa4}, one can conclude that
\begin{equation}
\nu\wedge TU =2\mu\nu\wedge\mathscr{D}_\nu U-\mu\nu\wedge(\nu\wedge (\nabla\wedge U)) =0 \quad\mbox{on}\ \ \Gamma.
\end{equation}

The proof is complete.
\end{proof}

In a similar manner to Proposition~\ref{prop:decom}, by applying Theorem~\ref{thm:thirdkind}, we have the following result.

\begin{prop}\label{prop:decom3}
Let $D=[D,\mathrm{III}]$ be a third kind elastic scatterer with a piecewise flat boundary; that is, the geometric condition \eqref{eq:gc32} is fulfilled. Let $U([D,\mathrm{III}])=U(\cdot; U^{in}, [D, \mathrm{III}])$ with the incident field $U^{in}$ given by \eqref{eq:planewave}. 
	
Set
\begin{equation}\label{eq:dsp212}
	v_p=v_p([D,\mathrm{III}]):=-\nabla\cdot U([D,\mathrm{III}]),\quad v_p^{sc}:=-\nabla\cdot U^{sc}([D,\mathrm{III}]).
\end{equation}
Then $v_p$ is a solution to the following problem,
\begin{equation}\label{eq:dsp332}
\begin{cases}
& (\Delta+k_p^2) v_p=0\qquad\mbox{in}\ \ \mathbb{R}^3\backslash\overline{D},\medskip\\
& v_p=v_p^{in}+v_p^{sc}\qquad \mbox{in}\ \ \mathbb{R}^3\backslash\overline{D},\medskip\\
& \nu\cdot\nabla v_p=0\qquad\mbox{on}\ \ \partial D,\medskip\\
& \displaystyle{\frac{x}{\|x\|}\cdot(\nabla v_p^{sc})-ik_p v_p^{sc}=\mathcal{O}(\|x\|^{-2})}\ \mbox{as}\ \|x\|\rightarrow+\infty,
\end{cases}
\end{equation}
with $v_p^{in}$ given by \eqref{eq:dsp22}.
	
In parallel, if we set
\begin{equation}\label{eq:dsp5half112}
	\begin{split}
		& E_s=E_s([D,\mathrm{III}]):=\nabla\wedge U([D,\mathrm{III}]),\  E_s^{sc}:=\nabla\wedge U^{sc}([D,\mathrm{III}]),\\
		& H_s=H_s([D,\mathrm{III}]):=\nabla\wedge E_s/(ik_s), 
	\end{split}
\end{equation}
then $(E_s,H_s)$ solves the following problem,
\begin{equation}\label{eq:dsp612}
	\begin{cases}
		& \nabla\wedge E_s-ik_s H_s=0\qquad \mbox{in}\ \ \mathbb{R}^3\backslash\overline{D},\medskip\\
		& \nabla\wedge H_s+ik_s E_s=0\qquad \mbox{in}\ \ \mathbb{R}^3\backslash\overline{D},\medskip\\
		& E_s=E_s^{in}+E_s^{sc}\qquad \mbox{in}\ \ \mathbb{R}^3\backslash\overline{D},\medskip\\
		& \nu\wedge E_s =0\qquad \mbox{on}\ \ \partial D,\medskip\\
		& \displaystyle{\frac{x}{\|x\|}\wedge (\nabla\wedge E^{sc}_s) + ik_sE^{sc}_s =\mathcal{O}(\|x\|^{-2})}\ \mbox{as}\ \|x\|\rightarrow+\infty, %E_s^{sc}\ \mbox{satisfies {\color{red} the radiation condition \eqref{eq:radiations}}},
	\end{cases}
\end{equation}
with $E_s^{in}$ given by \eqref{eq:dsp5half12}.
\end{prop}

The Helmholtz system \eqref{eq:dsp332} is known as the scattering for acoustic waves from a \emph{sound-hard} scatterer. Hence, $[D,\mathrm{III}]$ can be regarded as a sound-hard scatterer, which shall be denoted as $[D,\mathrm{SH}]$ in this situation. 
Moreover, in parallel to \eqref{eq:dsp61}, the Maxwell system \eqref{eq:dsp612} describes the scattering for electromagnetic waves by a PEC scatterer $[D,\mathrm{PEC}]=[D,\mathrm{III}]$. 

In what follows, similar to \eqref{eq:ind11} and \eqref{eq:ind2f1}, the correspondences between $U(\cdot; U^{in}, [D, \mathrm{III}])$ and the solutions to the problems \eqref{eq:dsp332} and \eqref{eq:dsp612} shall be written as
\begin{equation}\label{eq:ind112}
\begin{split}
-\nabla\cdot U(U^{in}; [D,\mathrm{III}])=& v_p(v_p^{in}; [D,\mathrm{SH}]),\\
U_p(U^{in}; [D,\mathrm{III}])=& \frac{1}{k_p^2}\nabla v_p(v_p^{in}; [D,\mathrm{SH}]),
\end{split}
\end{equation}
and
\begin{equation}\label{eq:ind2f12}
\begin{split}
\nabla\wedge U(U^{in}; [D,\mathrm{III}])=& E_s(E_s^{in}; [D,\mathrm{PEC}]),\\
U_s(U^{in}; [D,\mathrm{III}])=& \frac{1}{k_s^2}\nabla\wedge E_s(E_s^{in}; [D,\mathrm{PEC}]). 
\end{split}
\end{equation}

Next, we show that the geometric condition \eqref{eq:gc32} on the scatterer $[D,\mathrm{III}]$ is also  ``necessary" in a certain sense for the decoupling of the pressure and shear waves. In fact, we can show

\begin{thm}\label{thm:thirdkind2}
Let $D$ be a third kind scatterer. Suppose that there holds
\begin{equation}\label{eq:gc1n131}
K(x)\neq 0\quad\mbox{for a.e.}\ x\in\partial D.
\end{equation}
Let $U$ be a solution to the Lam\'e system \eqref{eq:lames}--\eqref{eq:radiation}. Suppose further that
\begin{equation}\label{eq:gc1n232}
U(\cdot; U^{in}, [D, \mathrm{III}])\equiv\hspace*{-3.5mm}\backslash\ U(\cdot; U^{in}, [D, \mathrm{I}])\quad\mbox{in}\ \mathbb{R}^3\backslash\overline{D}.
\end{equation}
Then \eqref{eq:thirdkindga} and \eqref{eq:decompps232} cannot hold true simultaneously on $\partial D$.
\end{thm}

\begin{proof}
We prove the theorem by an absurdity argument. Assume contrarily that both \eqref{eq:thirdkindga} and \eqref{eq:decompps232} hold on $\partial D$ for the scattering system \eqref{eq:lames}--\eqref{eq:radiation}. Then we have by from $\nu\wedge TU=0$ and $\nu\wedge(\nabla\wedge U)=\nu\wedge(\nabla\wedge U_s)=0$ on $\partial D$ that
\begin{equation}\label{eq:q1}
2\mu\nu\wedge \partial_\nu U=\nu\wedge TU-\mu\nu\wedge\big(\nu\wedge(\nabla\wedge U) \big)=0\quad\mbox{on}\ \ \partial D,
\end{equation}
which together with \eqref{eq:pp37} readily implies that
\begin{equation}\label{eq:q2}
\nu\wedge\mathscr{D}_\nu U=0\quad\mbox{on}\ \ \Gamma,
\end{equation}
where $\mathscr{D}_\nu$ is defined in \eqref{eq:pp32}. Next, by using $\nu\cdot U=0$ on $\partial D$ in
\eqref{eq:thirdkindga}, we have
\begin{equation}\label{eq:q3}
\begin{split}
0=&\nu\wedge\mathrm{Grad}_\Gamma(\nu\cdot U)\\
=& \nu\wedge\sum_{j=1}^3\nu_j\nabla U_j+\nu\wedge\sum_{j=1}^3 U_j
\mathrm{Grad}_\Gamma\nu_j\\
=& \nu\wedge\mathscr{D}_\nu U++\nu\wedge\sum_{j=1}^3 U_j
\mathrm{Grad}_\Gamma\nu_j,
\end{split}
\end{equation}
which together with \eqref{eq:q2} further yields that
\begin{equation}\label{eq:q4}
\nu\wedge\sum_{j=1}^3 U_j
\mathrm{Grad}_\Gamma\nu_j=0\quad\mbox{on}\ \ \partial D.
\end{equation}

Let $(g_{kl})_{k,l=1}^{2}$ be the first fundamental matrix of the differential geometry for $\Gamma$, see \eqref{eq:fdmm}.
Using \eqref{eq:sg1}, one can verify directly that
\begin{equation}
\begin{split}
\sqrt{|g|} \, \nu\wedge \mathrm{Grad}_\Gamma \nu_j &= \left( \frac{\partial x}{\partial u_1} \wedge \frac{\partial x}{\partial u_2} \right)  \wedge \sum_{k,l=1}^{2} g^{kl} \frac{\partial \nu_l}{\partial u_k} \frac{\partial x}{\partial u_l}\\
& = \sum_{k,l=1}^{2} g^{kl} \frac{\partial \nu_j}{\partial u_k} \left( g_{l1}\,\frac{\partial x}{\partial u_2} -g_{l2}\,\frac{\partial x}{\partial u_1} \right) \\
& = \sum_{k=1}^{2}  \frac{\partial \nu_j}{\partial u_k} \left( \delta_{k,1}\,\frac{\partial x}{\partial u_2} -\delta_{k,2}\,\frac{\partial x}{\partial u_1} \right)\\
& = \frac{\partial \nu_j}{\partial u_1}  \frac{\partial x}{\partial u_2} -\frac{\partial \nu_j}{\partial u_2} \frac{\partial x}{\partial u_1}, \quad j=1,2,3.
\end{split}
\end{equation}
Thus, we have
\begin{equation}\label{eq:pp331}
\begin{split}
& \nu\wedge\sum_{j=1}^3 U_j\mathrm{Grad}_\Gamma \nu_j \\
= & \frac{1}{\sqrt{|g|}} \sum_{j=1}^3 \left( U_j \frac{\partial \nu_j}{\partial u_1}  \frac{\partial x}{\partial u_2} - U_j\frac{\partial \nu_j}{\partial u_2} \frac{\partial x}{\partial u_1}\right) \\
= & - \frac{1}{\sqrt{|g|}}   \left( U \cdot \frac{\partial \nu}{\partial u_2} \right) \frac{\partial x}{\partial u_1} + \frac{1}{\sqrt{|g|}}   \left( U \cdot \frac{\partial \nu}{\partial u_1} \right) \frac{\partial x}{\partial u_2} .
\end{split}
\end{equation}
In view of \eqref{eq:q4} and \eqref{eq:pp331}, one readily has
\begin{equation}\label{eq:eqthm24_3}
U \cdot \frac{\partial \nu}{\partial u_l}=0\quad\mbox{on}\ \ \partial D,\quad l=1,2.
\end{equation}

Next, we show that the geometric condition \eqref{eq:gc1n131} implies that
\begin{equation}\label{eq:gc1n1311}
\frac{\partial \nu}{\partial u_1} \wedge \frac{\partial \nu}{\partial u_2} \neq 0\quad\mbox{for a.e.}\ x\in\partial D.
\end{equation}
Indeed, one can represent the tangential vectors $\partial \nu/\partial u_l$, $l=1,2$ as
\begin{equation}\label{eq:rem22_11}
\frac{\partial \nu}{\partial u_l}=\sum_{k=1}^{2} - b_l^k \, \frac{\partial x}{\partial u_k},
\end{equation}
from which we can directly compute that
\begin{equation}\label{eq:q5}
\begin{split}
& \frac{\partial \nu}{\partial u_1} \wedge \frac{\partial \nu}{\partial u_2} = \sum_{k,l=1}^{2}  b_1^k b_2^l \, \frac{\partial x}{\partial u_k}  \wedge \frac{\partial x}{\partial u_l} \\
= & \left( b^1_1 b^2_2 - b_1^2 b_2^1 \right) \frac{\partial x}{\partial u_1}  \wedge \frac{\partial x}{\partial u_2}=K(x)  \frac{\partial x}{\partial u_1}  \wedge \frac{\partial x}{\partial u_2}.
\end{split}
\end{equation}
Clearly, \eqref{eq:q5} implies that \eqref{eq:gc1n1311} holds true. That means the three vectors $\nu$, $\partial\nu/\partial u_1$ and $\partial \nu/\partial\nu_2$ are linearly independent a.e. on $\partial D$. Hence, by virtue of \eqref{eq:eqthm24_3} and  $\nu\cdot U=0$ on $\partial D$ in \eqref{eq:thirdkindga}, one immediately has
\begin{equation}\label{eq:q6}
U=0\quad\mbox{on}\ \ \partial D.
\end{equation}
That is, $D$ is also a first kind scatterer and hence $U(\cdot; U^{in}, [D, \mathrm{III}])=U(\cdot; U^{in}, [D, \mathrm{I}])\ \mbox{in}\ \mathbb{R}^3\backslash\overline{D}$, which is a contradiction to \eqref{eq:gc1n232}.

The proof is complete.
\end{proof}

\section{Elastic scattering from a ball}\label{sect:3}

In this section, we consider the elastic scattering from a ball; that is, we assume that $D=B_R$. In this case, one has
\begin{equation}\label{eq:cvtb}
H(x)\equiv -\frac{1}{R}\quad\mbox{and}\quad K(x)\equiv \frac{1}{R^2}.
\end{equation}
Hence, both the geometric conditions \eqref{eq:gcmc} and \eqref{eq:gc32} are not satisfied. We shall show that the pressure and shear waves cannot be decoupled, which reinforce our theoretical results in Section~\ref{sect:2}.

\subsection{Spherical wave functions}

Our arguments rely essentially on the expansion method with the spherical wave functions. We first collect some preliminary knowledge for the subsequent use, and refer to \cite{CK12,Ned} for more related results.

Let $j_n(t)$ and $h_n(t)$, $n\in\mathbb{N}\cup\{0\}$ and $t\in\mathbb{R}$, be the spherical Bessel function and the spherical Hankel function of the first kind, respectively.
Let $\{Y_n^m(\hat x);m=-n,\ldots,n\}$ be the orthonormal system of spherical harmonics of order $n\in\mathbb{N}$. %introduced in \cite[Section 2.3]{CoK12}.
The tangential fields on the unit sphere $\mathbb{S}^2$
\[
\frac{1}{\sqrt{n(n+1)}}\,\mathrm{Grad}\,Y_n^m(\hat x)\quad\mbox{and}\quad \frac{1}{\sqrt{n(n+1)}}\hat{x}\wedge\mathrm{Grad}\, Y_n^m(\hat x),
\]
for $m=-n,\ldots,n$, $n=1,2,\ldots$, are called vectorial spherical harmonics and form a complete orthonormal system in $L^2_t(\mathbb{S}^2)$ (cf. \cite{CK12}).

For $n\in\mathbb{N}\cup\{0\}$, $m=-n,\ldots,n$, and $f\equiv j$ or $f\equiv h$, we introduce the scalar functions
\begin{equation}\label{eq:eq2}
u_n^m[f;k](x):=f_n(k\|x\|)Y_n^m(\hat{x}),
\end{equation}
and the vector-valued functions
\begin{equation}\label{eq:eq3}
M_n^m[f;k](x):=\nabla\wedge\Big(xf_n(k\|x\|)Y_n^m(\hat{x})\Big).
\end{equation}
It is known that (cf. \cite{CK12}), $u_n^m[j;k_p]$ (resp. $u_n^m[h;k_p]$) is an entire (resp., a radiating) solution of the Helmholtz equation
\begin{equation}\label{eq:dsp31}
(\Delta+k_p^2) v=0,
\end{equation}
and $M_n^m[j;k_s]$ and $\nabla\wedge M_n^m[j]$ (resp. $M_n^m[h; k_s]$ and $\nabla\wedge M_n^m[h; k_s]$) are entire (resp., radiating) solutions of the Maxwell equation
\begin{equation}\label{eq:dsp611}
(\Delta+k_s^2) E=0,\quad  \nabla\cdot E =0.
\end{equation}

By direct calculations, one has for $n\in\mathbb{N}$, $m=-n,\ldots,n$ and $f\equiv j$ or $h$ that
\begin{equation}\label{eq:eq4}
\nabla u_n^m[f;k] (x) =kf'_n(k\|x\|)Y_n^m(\hat{x})\hat{x}+\frac{1}{\|x\|}f_n(k\|x\|)\mathrm{Grad}\, Y_n^m(\hat{x}),
\end{equation}
\begin{equation}\label{eq:eq5}
 M_n^m[f;k](x)=f_n(k\|x\|)\mathrm{Grad}\, Y_n^m(\hat{x})\wedge \hat{x},
\end{equation}
\begin{equation}\label{eq:eq6}
 \nabla\wedge\Big(M_n^m[f;k](x)\Big)=  n(n+1)\frac{f_n(k\|x\|)}{\|x\|}Y_n^m(\hat{x})\hat{x}+\frac{\check{f}_n(k\|x\|)}{\|x\|}\mathrm{Grad}\, Y_n^m(\hat{x}),
\end{equation}
where the function and $\check{f}_n$ is defined by
\[
\check{f}_n(t):=f_n(t)+tf'_n(t).
\]
In what follows, when there is no ambiguity, we omit the letter $k$ from the notations $u_n^m[f;k]$ and $M_n^m[f;k]$ and simply write them as $u_n^m[f]$ and $M_n^m[f]$, respectively.

Recall from \cite{CK12} that any radiating solutions $v_p^{sc}$ to the Helmholtz equation \eqref{eq:dsp31} and $E_s^{sc}$ to the Maxwell equations \eqref{eq:dsp611} for $\|x\|>R>0$ have the following expansions
\begin{equation}\label{eq:3}
v_p^{sc}=\sum_{n=0}^{\infty}\sum_{m=-n}^{n}a_n^m u_n^m[h;k_p],
\end{equation}
and
\begin{equation}\label{eq:4}
E_s^{sc}= \sum_{n=1}^{\infty}\sum_{m=-n}^{n}\Big ( b_n^m \nabla\wedge M_n^m[h; k_s] +c_n^m M_n^m[h; k_s] \Big ),
\end{equation}
which (along with their respective derivatives) converge uniformly on compact subsets of $\|x\|>R$.
Correspondingly, the far field patterns for the radiating solutions $v_p^{sc}$ and $E_s^{sc}$ of the expansions \eqref{eq:3} and \eqref{eq:4}, respectively, are given by
\begin{equation}\label{eq:eqq1}
v_p^{\infty}(\hat{x})=\frac{1}{k_p}\sum_{n=0}^{\infty}\sum_{m=-n}^{n} \frac{1}{i^{n+1}}a_n^m Y_n^m(\hat{x}),
\end{equation}
and
\begin{equation}\label{eq:eqq2}
  E_s^{\infty}(\hat{x})=\frac{1}{k_s}\sum_{n=1}^{\infty}\sum_{m=-n}^{n} \frac{1}{i^{n+1}}\Big( ik_sb_n^m \mathrm{Grad}Y_n^m(\hat{x})-c_n^m \hat{x}\wedge\mathrm{Grad}Y_n^m(\hat{x})\Big).
\end{equation}

Let
\[
\Phi(x,y):=\frac{1}{4\pi}\frac{e^{ik\|x-y\|}}{\|x-y\|}
\]
be the fundamental solution to the PDO $-(\Delta+k^2)$. For any vector $p\in\mathbb{R}^3$, we have the following expansion (see for instance, \cite[Theorem 6.29]{CK12}):
\begin{equation}\label{eq:5}
  \begin{split}
    p\,\Phi(z,y)= & ik\sum_{n=1}^{\infty}\sum_{m=-n}^{n}\frac{1}{n(n+1)}p\cdot\overline{M_n^m[j](y)}\, M_n^m[h](z) \\
      & +\frac{i}{k}\sum_{n=1}^{\infty}\sum_{m=-n}^{n}\frac{1}{ n(n+1)}p\cdot\nabla\wedge\overline{M_n^m[j](y)}\, \nabla\wedge M_n^m[h](z)\\
      &{+\frac{i}{k}\sum_{n=0}^{\infty}\sum_{m=-n}^{n}p\cdot\nabla \overline{u_n^m[j](y)} \,\nabla u_n^m[h](z),}
  \end{split}
\end{equation}
which is uniformly convergent on any compact subset of $\|z\|>\|y\|$.

We derive the following lemma for the later use.

\begin{lem}
Let $d\in\mathbb{S}^2$ and $p\in\mathbb{R}^3$ be two vectors such that $p\cdot d=0$. Then one has
\begin{equation}\label{eq:eqq5}
  \begin{split}
    p\,e^{ikx\cdot d}= & 4\pi\sum_{n=1}^{\infty}\sum_{m=-n}^{n}\frac{i^n}{n(n+1)}p\cdot (\mathrm{Grad}\, \overline{Y_n^m(d)}\wedge d)\,M_n^m[j](x)  \\
      & -4\pi\frac{1}{k}\sum_{n=1}^{\infty}\sum_{m=-n}^{n} \frac{i^{n+1}}{n(n+1)}p\cdot\mathrm{Grad}\, \overline{Y_n^m(d)}\,\nabla\wedge M_n^m[j](x),
  \end{split}
\end{equation}
and
\begin{equation}\label{eq:eqq4}
    de^{ikx\cdot d}= -4\pi\frac{1}{k}\sum_{n=0}^{\infty}\sum_{m=-n}^{n}i^{n+1}\overline{Y_n^m(d)}\,\nabla\,u_n^m[j](x).
\end{equation}
\end{lem}

\begin{proof}
The proof basically follows the strategies in \cite[Chapter 2]{CK12}.
For any $r\in\mathbb{R}_+$, we set $z=r\hat{x}$ and $y=-\|x\|d$. Notice that
\[
Y_n^m(-d)=(-1)^nY_n^m(d)\quad\text{and}\quad \mathrm{Grad}\, Y_n^m(-d)=(-1)^{n+1}\mathrm{Grad}\, Y_n^m(d).
\]
From \eqref{eq:eq4}-\eqref{eq:eq6}, we have
\begin{equation}\label{eq:eqq3}
  \begin{array}{ccc}
  p\cdot\overline{M_n^m[j](y)}&=&(-1)^nj_n(k\|z\|)p\cdot(\mathrm{Grad}\, \overline{Y_n^m(d)}\wedge d),\medskip\\
  p\cdot\nabla\wedge\overline{M_n^m[j](y)}&=&(-1)^{n+1}\frac{1}{\|z\|}\check{j}_n(k\|z\|)p\cdot\mathrm{Grad}\,\overline{Y_n^m(d)},\medskip\\
  p\cdot\nabla\,\overline{u_n^m[j](y)}&=&(-1)^{n+1}\frac{1}{\|z\|}j_n(k\|z\|)p\cdot\mathrm{Grad}\,\overline{Y_n^m(d)}.
  \end{array}
\end{equation}
Passing to the limit $r\rightarrow+\infty$ in the expansion \eqref{eq:5} for $\Phi(z,y)$, inserting \eqref{eq:eqq3}, and making use of the relations \eqref{eq:eqq1} and \eqref{eq:eqq2}, one can derive that
\[
  \begin{split}
    \frac{1}{4\pi}p\,e^{ikx\cdot d}= & \sum_{n=1}^{\infty}\sum_{m=-n}^{n}\frac{i^n}{n(n+1)}p\cdot(\mathrm{Grad}\,\overline{Y_n^m(d)}\wedge d)j_n(k\|z\|)\mathrm{Grad}\,Y_n^m(\hat{z})\wedge\hat{z}  \\
      & -\frac{1}{k}\sum_{n=1}^{\infty}\sum_{m=-n}^{n} \frac{i^{n+1}}{n(n+1)}p\cdot\mathrm{Grad}\,\overline{Y_n^m(d)}\frac{\check{j}_n(k\|z\|)}{\|z\|}\mathrm{Grad}\,Y_n^m(\hat{z})\\
      & -\frac{1}{k}\sum_{n=1}^{\infty}\sum_{m=-n}^{n} i^{n+1}p\cdot \mathrm{Grad}\,\overline{Y_n^m(d)}\frac{j_n(k\|z\|)}{\|z\|}Y_n^m(\hat{z})\hat{z},
  \end{split}
\]
which immediately yields \eqref{eq:eqq5}.

Replacing the vector $p$ by $d$ and applying similar arguments as above, one can obtain \eqref{eq:eqq4}.
The series \eqref{eq:eqq4} can also be obtained straightforwardly by taking gradient on both sides of the expansion (see for instance, \cite[Chapter 2]{CK12}
\[
e^{ikx\cdot d}=4\pi\sum_{n=0}^{\infty}\sum_{m=-n}^{n}i^n \overline{Y_n^m(d)}j_n(k\|x\|)Y_n^m(\hat{x}).
\]

The proof is complete.
\end{proof}

\subsection{Elastic scattering from a third or fourth kind ball}

Let $U^{in}=U_p^{in}+U_s^{in}$ be the incident plane elastic wave given by the general form \eqref{eq:planewave}.
In the rest of this section, we shall always let $\nu(x)=\hat{x}$ be the outward unit normal vector to the central ball $B_R$.

For $n\in\mathbb{N}\cup\{0\}$ and $m=-n,\ldots,n$, we define
\[
A_n^m(1):=4\pi i^{n+1}\overline{Y_n^m(d)},\quad A_n^m(2):=4\pi i^{n+1}d^{\perp}\cdot\mathrm{Grad}\,\overline{Y_n^m(d)},
\]
and
\[
A_n^m(3):=4\pi i^n d^{\perp}\cdot\left(\mathrm{Grad}\,\overline{Y_n^m(d)}\wedge d\right).
\]
It is directly computed from \eqref{eq:eqq4} and \eqref{eq:eq4} for $ U_p^{in}=\alpha_p  d e^{ik_pd\cdot x }$ that
\begin{equation}\label{eq:eq3161}
\begin{array}{rl}
\hat{x}\cdot U_p^{in}=& \displaystyle{ -\alpha_p \sum_{n=0}^{\infty}\sum_{m=-n}^{n}A_n^m(1)j'_n(k_p\|x\|)Y_n^m(\hat{x})},\medskip\\
\hat{x}\wedge U_p^{in}= & \displaystyle{- \alpha_p\sum_{n=0}^{\infty}\sum_{m=-n}^{n}A_n^m(1)\frac{j_n(k_p\|x\|)}{k_p\|x\|}\hat{x}\wedge\mathrm{Grad}\, Y_n^m(\hat{x}),}
\end{array}
\end{equation}
and that
\begin{equation}\label{eq:eq314}
\nabla\cdot U_p^{in}=\alpha_p \nabla\cdot \left( d e^{ik_pd\cdot x }\right) =\alpha_p k_p\sum_{n=0}^{\infty}\sum_{m=-n}^{n} A_n^m(1) u_n^m[j; k_p],
\end{equation}
We also have from \eqref{eq:eqq5}, \eqref{eq:eq5} and \eqref{eq:eq6} for $U_s^{in}=\alpha_s  d^{\perp} e^{ik_sd\cdot x }$ that
\begin{equation}\label{eq:eq3162}
\begin{array}{rl}
\hat{x}\cdot U_s^{in}=& \displaystyle{- \alpha_s\sum_{n=1}^{\infty}\sum_{m=-n}^{n}A_n^m(2)\frac{j_n(k_s\|x\|)}{k_s\|x\|}Y_n^m(\hat{x}),}\medskip \\
\hat{x}\wedge U_s^{in} = & \displaystyle{\alpha_s\sum_{n=1}^{\infty}\frac{1}{n(n+1)}\sum_{m=-n}^{n} A_n^m(3)j_n(k_s\|x\|)\mathrm{Grad}\,Y_n^m(\hat{x})}\medskip \\
&\hspace*{-8mm}\displaystyle{ -\alpha_s\sum_{n=1}^{\infty}\frac{1}{n(n+1)}\sum_{m=-n}^{n}A_n^m(2)\frac{\check{j}_n(k_s\|x\|)}{k_s\|x\|}\hat{x}\wedge\text{Grad}\,Y_n^m(\hat{x}).}
\end{array}
\end{equation}
%\begin{equation}\label{eq:eq315}
%  \begin{array}{rl}
%  \hat{x}\cdot U_p^{in}=& \displaystyle{ -\alpha_p \sum_{n=0}^{\infty}\sum_{m=-n}^{n}A_n^m(1)j'_n(k_p\|x\|)Y_n^m(\hat{x})},\medskip\\
%  \hat{x}\cdot U_s^{in}=& \displaystyle{- \alpha_s\sum_{n=1}^{\infty}\sum_{m=-n}^{n}A_n^m(2)\frac{j_n(k_s\|x\|)}{k_s\|x\|}Y_n^m(\hat{x}),}
%  \end{array}
%\end{equation}
%and
%\begin{equation}\label{eq:eq316}
%  \begin{array}{rl}
% \hat{x}\wedge U_p^{in}= & \displaystyle{- \alpha_p\sum_{n=0}^{\infty}\sum_{m=-n}^{n}A_n^m(1)\frac{j_n(k_p\|x\|)}{k_p\|x\|}\hat{x}\wedge\mathrm{Grad}\, Y_n^m(\hat{x}),}\medskip\\
% \hat{x}\wedge U_s^{in} = & \displaystyle{\alpha_s\sum_{n=1}^{\infty}\frac{1}{n(n+1)}\sum_{m=-n}^{n} A_n^m(3)j_n(k_s\|x\|)\mathrm{Grad}\,Y_n^m(\hat{x})}\medskip \\
%    &\hspace*{-8mm}\displaystyle{ -\alpha_s\sum_{n=1}^{\infty}\frac{1}{n(n+1)}\sum_{m=-n}^{n}A_n^m(2)\frac{\check{j}_n(k_s\|x\|)}{k_s\|x\|}\hat{x}\wedge\text{Grad}\,Y_n^m(\hat{x}).}
%  \end{array}
%\end{equation}

It is obtained by \eqref{eq:3} and \eqref{eq:4} that, for any radiating solution $U^{sc}$ of the Navier equation \eqref{eq:lames} outside a certain ball, $U^{sc}=U^{sc}_p+U^{sc}_s$ has the following series expansions
\begin{equation}\label{eq:eqqq5}
  \begin{array}{rl}
    U^{sc}_p =& \displaystyle{\sum_{n=0}^{\infty}\sum_{m=-n}^{n}a_n^m \nabla u_n^m[h; k_p],}\medskip \\
    U^{sc}_s =& \displaystyle{\sum_{n=1}^{\infty}\sum_{m=-n}^{n}\Big( b_n^m \nabla\wedge M_n^m[h;k_s] +c_n^m M_n^m[h; k_s]\Big).}
  \end{array}
\end{equation}
Using the series representations in \eqref{eq:eqqq5}, it is directly computed that
\begin{equation}\label{eq:eq318}
\nabla\cdot U^{sc}_p =-k_p^2\sum_{n=0}^{\infty}\sum_{m=-n}^{n}a_n^m u_n^m[h; k_p],
\end{equation}
and
\begin{equation}\label{eq:eq319}
  \begin{array}{rl}
    \hat{x}\cdot U^{sc}_p =& \displaystyle{k_p \sum_{n=0}^{\infty}\sum_{m=-n}^{n}a_n^m h'_n(k_p \|x\|) Y_n^m(\hat{x}),} \\
    \hat{x}\cdot U^{sc}_s =& \displaystyle{ \frac{1}{\|x\|}\sum_{n=1}^{\infty}\sum_{m=-n}^{n} n(n+1) b_n^m h_n(k_s\|x\|)Y_n^m(\hat{x}),}
  \end{array}
\end{equation}
and
\begin{equation}\label{eq:eq320}
  \begin{array}{rl}
    \hat{x}\wedge U_p^{sc} =& \displaystyle{\sum_{n=0}^{\infty}\sum_{m=-n}^{n}a_n^m \frac{h_n(k_p \|x\|)}{\|x\|}\hat{x}\wedge\mathrm{Grad}\, Y_n^m(\hat{x}),} \\
    \hat{x}\wedge U_s^{sc} =& \displaystyle{\sum_{n=1}^{\infty}\sum_{m=-n}^{n} c_n^m h_n(k_s \|x\|)\mathrm{Grad}\, Y_n^m(\hat{x})}\\
     & \displaystyle{+\sum_{n=1}^{\infty}\sum_{m=-n}^{n} b_n^m\frac{\check{h}_n(k_s \|x\|)}{\|x\|} \hat{x}\wedge\mathrm{Grad}\, Y_n^m(\hat{x}).}
  \end{array}
\end{equation}

The following lemma shall also be needed.

\begin{lem}\label{lem:2}
	Let $\mathscr{D}_\nu$ be defined in \eqref{eq:pp32} with $\nu(x)= \hat{x}$. Then there holds
	\[
	\mathscr{D}_{\nu}U(x)=\nabla(\nu\cdot U)(x)-\frac{1}{\|x\|}U(x)+\frac{1}{\|x\|}\nu(x)(\nu\cdot U)(x),\ \ x\in\partial B_R.
	\]
\end{lem}

\begin{proof}
	It is clear that
	\[
	\nu\cdot \partial_j U=\partial_j(\nu\cdot U)-U\cdot \partial_j\nu,\quad j=1,2,3.
	\]
	The proof can be concluded by noting for $j=1,2,3$ that
	\[
	\partial_j\nu(x)=\frac{1}{\|x\|}\mathbf{e}_j-\frac{x_j}{\|x\|^2}\nu.
	\]
\end{proof}

\subsubsection{The fourth kind boundary condition}\label{sect:ball_4th}

Now we consider the case that the ball $B_R$ is a fourth kind scatterer. By using the series expansions introduced earlier, we determine the scattered field $U^{sc}=U^{sc}_p+U^{sc}_s$ in the form \eqref{eq:eqqq5}, such that the fourth kind boundary condition \eqref{eq:fourthkind} holds for $U=U^{in}+U^{sc}$ on $\partial B_R$, namely,
\[
\nu\wedge U^{sc}=-\nu\wedge U^{in}\quad\text{and}\quad\nu\cdot TU^{sc}=-\nu\cdot TU^{in}\quad\mbox{on $\partial B_R$}.
\]

For the condition $\nu\wedge U=0$ on $\partial B_R$, one first sees from \eqref{eq:eq3161}, \eqref{eq:eq3162} and \eqref{eq:eq320} that
\begin{equation}\label{eq:eqqqq2}
  c_n^m=-\alpha_s\frac{A_n^m(3)}{n(n+1)}\frac{j_n(k_sR)}{h_n(k_sR)}.
\end{equation}
On the other hand, from Lemma~\ref{lem:2} and the equation \eqref{eq:pp37}, it is verified that
\begin{equation}\label{eq:eqq6}
\begin{split}
 \nu\cdot TU & =2\mu\, \nu\cdot \mathscr{D}_{\nu}U+\lambda \nabla\cdot U\\
 &=2\mu\, \nu\cdot\nabla(\nu\cdot U)+\lambda \nabla\cdot U.
\end{split}
\end{equation}
Inserting \eqref{eq:eq314} and \eqref{eq:eq3161} into \eqref{eq:eqq6}, we then arrive at
\[
\begin{split}
 \nu\cdot TU_p^{in} & =2\mu\, \nu\cdot\nabla (\nu\cdot U_p^{in})+\lambda \nabla\cdot U_p^{in} \\
    & =a_p \sum_{n=0}^{\infty}\sum_{m=-n}^{n} A_n^m(1)k_p \Big(\lambda j_n(k_p\|x\|)-2\mu j''_n(k_p \|x\|) \Big) Y_n^m(\hat{x}).
\end{split}
\]
In the same way, applying \eqref{eq:eq3162} to \eqref{eq:eqq6}, we have
\[
 \nu\cdot TU_s^{in}  =2\mu\, \nu\cdot\nabla (\nu\cdot U_s^{in})  = \alpha_s\sum_{n=1}^{\infty}\sum_{m=-n}^{n}A_n^m(2)2\mu\frac{\widetilde{j}_n(k_s\|x\|)}{k_s \|x\|^2}Y_n^m(\hat{x}),
\]
where the functions $\widetilde{j}_n$, $n\in\mathbb{N}$, are given by $\widetilde{j}_n(t):=j_n(t)-tj'_n(t)$.
Similarly, we have from \eqref{eq:eq318}, \eqref{eq:eq319} and \eqref{eq:eqq6} that
\[
\nu\cdot TU_s^{sc}=2\mu\,\hat{x}\cdot\nabla\left(\hat{x}\cdot U_s^{sc}\right)=-2\mu\sum_{n=1}^{\infty}\sum_{m=-n}^{n} b_n^m n(n+1) \frac{\widetilde{h}_n(k_s \|x\|)}{k_s \|x\|^2}Y_n^m(\hat{x}),
\]
and
\[
\begin{split}
  \nu\cdot TU_p^{sc} & = 2\mu\, \hat{x}\cdot\nabla\left(\hat{x}\cdot U_p^{sc}\right)+\lambda\nabla\cdot U_p^{sc}\\
    & =-\sum_{n=0}^{\infty}\sum_{m=-n}^{n} a_n^m k_p^2 \Big(\lambda h_n(k_p \|x\|)-2\mu h''_n(k_p \|x\|)\Big) Y_n^m(\hat{x}),
\end{split}
\]
with the functions $\widetilde{h}_n$, $n\in\mathbb{N}$, defined by $\widetilde{h}_n(t):=h_n(t)-th'_n(t)$.

If both \eqref{eq:fourthkindga} and \eqref{eq:decompps2} hold for the scatterer $[B_R, \mathrm{IV}]$, then one should have $\nu\wedge U_p=\nu\wedge U_s=0$ on $\partial B_R$. As a consequence, it is deduced from \eqref{eq:eq314} and \eqref{eq:eq320} that $a_n^m=\widetilde{a}_n^m$ and $b_n^m=\widetilde{b}_n^m$ with
\[
\widetilde{a}_n^m:=\alpha_pA_n^m(1)\frac{1}{k_p}\frac{j_n(k_pR)}{h_n(k_pR)}\quad\text{and}\quad \widetilde{b}_n^m=\alpha_s\frac{A_n^m(2)}{n(n+1)}\frac{1}{k_s}\frac{\check{j}_n(k_sR)}{\check{h}_n(k_sR)},
\]
for $n\in\mathbb{N}_+$ and $m=-n\ldots,n$.

Define $\widetilde{U}_p:=U^{in}_p+\widetilde{U}^{sc}_p$ and $\widetilde{U}_s:=U^{in}_s+\widetilde{U}^{sc}_s$ with
\begin{equation}\label{eq:eqqqq1}
  \begin{array}{rl}
    \widetilde{U}^{sc}_p :=&\displaystyle{ \sum_{n=1}^{\infty}\sum_{m=-n}^{n}\,\widetilde{a}_n^m \nabla u_n^m[h; k_p],}\medskip \\
    \widetilde{U}^{sc}_s :=& \displaystyle{ \sum_{n=1}^{\infty}\sum_{m=-n}^{n}\big(\widetilde{b}_n^m \nabla\wedge M_n^m[h; k_s] +c_n^mM_n^m[h;k_s]\big),}
  \end{array}
\end{equation}
where the constants $c_n^m$, $m=-n,\ldots,n$, $n\in\mathbb{N}$, are given by \eqref{eq:eqqqq2}.
Then $\widetilde{U}:=\widetilde{U}_p+\widetilde{U}_s$ is a solution to the exterior problem for the Navier equation \eqref{eq:lames}, which satisfies \eqref{eq:decompps2} on the boundary $\partial B_R$.

By using the property $$j_n(t) h'_n(t)-j'_n(t) h_n(t)=i/t^2$$ (cf. \cite{CK12,Ned}), one can directly compute for $x\in\partial B_R$ that
\[
\nu\cdot T\widetilde{U}_p(x)=-4\mu\frac{ i}{R^3} \sum_{n=1}^{\infty}\sum_{m=-n}^{n} \frac{a_p}{k_p^2} \frac{A_n^m(1)}{h_n(k_pR)}Y_n^m(\hat{x}),
\]
and that
\[
\nu\cdot T\widetilde{U}_s(x)=4\mu\frac{ i}{R^3} \sum_{n=1}^{\infty}\sum_{m=-n}^{n} \frac{\alpha_s}{k_s^2} \frac{A_n^m(2)}{h_n(k_sR)}Y_n^m(\hat{x}).
\]
Therefore, noting that $\nu\cdot TU=0$ on $\partial B_R$, in order for the decoupling \eqref{eq:decompps2} to be true for the elastic scattering problem with $[B_R,\mathrm{IV}]$, there should hold that
\begin{equation}%\label{}
  \frac{d^{\perp}\cdot\text{Grad}\overline{Y_n^m(d)}}{\overline{Y_n^m(d)}}=\frac{\alpha_p/k_p^2}{\alpha_s/k_s^2} \,\frac{h_n(k_sR)}{h_n(k_pR)},\quad m=-n,\ldots,n,
\end{equation}
for any $n\in\mathbb{N}$.
Clearly, this cannot be true for all $n\in\mathbb{N}$. Hence, \eqref{eq:fourthkindga} and \eqref{eq:decompps2} cannot hold simultaneously for the scatterer $[B_R, \mathrm{IV}]$. Moreover, by summarizing our earlier discussion, we have the following result.

\begin{prop}
For the scattered elastic field of the scattering problem \eqref{eq:lames}--\eqref{eq:radiation} corresponding to $[B_R,\mathrm{IV}]$, we have
\[
U^{sc}(\cdot;[B_R,IV])=\widetilde{U}^{sc}+\sum_{n=1}^{\infty}\sum_{m=-n}^{n} \alpha_n^m \Big(\nabla u_n^m[h; k_p]-\nabla\wedge M_n^m[h; k_s]\Big),
\]
where $\widetilde{U}^{sc}:=\widetilde{U}^{sc}_p+\widetilde{U}^{sc}_s$ is given by \eqref{eq:eqqqq1}, and $\alpha_n^m$, $m=-n,\ldots,n$, $n\in\mathbb{N}_+$ are constants satisfying
\begin{gather*}
  \alpha_n^m\left(h_n(t_p)-2n(n+1)\frac{\widetilde{h}_n(t_s)}{t_s^2}-2\frac{k_p^2}{k_s^2}\big( h_n(t_p)-h''_n(t_p)\big) \right) \\
  =\frac{4i}{k_s^2 R^3}\left( \frac{\alpha_s}{k^2_s}\,\frac{A_n^m(2)}{h_n(t_s)}- \frac{\alpha_p}{k^2_p}\,\frac{A_n^m(1)}{h_n(t_p)}\right),\quad m=-n,\ldots,n,\ \ n\in\mathbb{N},
\end{gather*}
with $t_p=k_p R$ and $t_s=k_s R$.

Particularly, $U(\cdot;[B_R,\mathrm{IV}])$ satisfies the decoupling property \eqref{eq:decompps2} if and only if $\alpha_n^m=0$ for all $m=-n,\ldots,n$ and $n\in\mathbb{N}$. This cannot be true unless $\alpha_s=\alpha_p=0$.
\end{prop}

\subsubsection{The third kind boundary condition}

In this subsection, we consider the case that $B_R$ is a third kind scatterer. Similar to our study in Section~\ref{sect:ball_4th}, we determine the scattered field $U^{sc}=U^{sc}_p+U^{sc}_s$ by using the series representation \eqref{eq:eqqq5}, such that the total wave $U=U^{in}+U^{sc}$ satisfies \eqref{eq:fourthkindga} on $\partial B_R$. Then we show that the decoupling property \eqref{eq:decompps232} does not hold for the scatterer $[B_R,\mathrm{III}]$.

In an analogous manner to our argument in Section~\ref{sect:ball_4th}, we first determine functions $\check{U}_p:=\check{U}_p^{sc}+U_p^{in}$ and $\check{U}_s:=\check{U}_s^{sc}-U_s^{in}$ with
\begin{equation}\label{eq:eq_checkU}
  \begin{array}{rl}
    \check{U}^{sc}_p :=&\displaystyle{ \sum_{n=0}^{\infty}\sum_{m=-n}^{n}\,\check{a}_n^m \text{grad}u_n^m[h;k_p],}\medskip \\
    \check{U}^{sc}_s :=& \displaystyle{\sum_{n=1}^{\infty}\sum_{m=-n}^{n}\left(\check{b}_n^m \nabla\wedge M_n^m[h;k_s] +\check{c}_n^mM_n^m[h;k_s]\right),}
  \end{array}
\end{equation}
such that $\check{U}=\check{U}_p+\check{U}_s$ %be the solution to the exterior problem for Navier equation in $\mathbb{R}^3\setminus \overline{B_R}$ who
satisfies the decoupling condition \eqref{eq:decompps232} on $\partial B_R$.

%It is obvious that $\text{curl} U_p^{in}=\text{curl} U_p^{sc}=0$, and one can obtain from \eqref{eq:eqq5} and \eqref{eq:eqqq5} that
%\[
%\text{curl} U_s^{in}=\alpha_s\sum_{n=1}^{\infty}\frac{1}{n(n+1)}\sum_{m=-n}^{n}\left\{A_n^m(3)\text{curl}M_n^m[j]-k_sA_n^m(2)M_n^m[j]\right\},
%\]
%and
%\[
%\text{curl} U_s^{sc}=\sum_{n=1}^{\infty}\sum_{m=-n}^{n}\left\{k_s^2 b_n^m M_n^m[h;k_s]+c_n^m \text{curl}M_n^m[h;k_s] \right\}.
%\]
It is directly computed that
\[
\begin{split}
  \nu\wedge \big(\nabla\wedge \check{U}_s^{sc}\big)= & k_s^2\sum_{n=1}^{\infty}\sum_{m=-n}^{n} \check{b}_n^m h_n(k_s \|x\|)\,\mathrm{Grad}\,Y_n^m(\hat{x}) \\
    & +\sum_{n=1}^{\infty}\sum_{m=-n}^{n} \check{c}_n^m\frac{\check{h}_n(k_s \|x\|)}{\|x\|}\,\hat{x}\wedge\mathrm{Grad}\, Y_n^m(\hat{x}).
\end{split}
\]
and obtained from \eqref{eq:eqq5} that
\begin{equation}\label{eq:eq326}
  \begin{split}
  \nu\wedge\big(\nabla\wedge U_s^{in}\big) =& \alpha_s\sum_{n=1}^{\infty}\sum_{m=-n}^{n} \frac{A_n^m(3)}{n(n+1)}\frac{\check{j}_n(k_s \|x\|)}{\|x\|}\hat{x}\wedge\mathrm{Grad}\,Y_n^m(\hat{x}) \\
    & -\alpha_s\sum_{n=1}^{\infty}\sum_{m=-n}^{n} \frac{A_n^m(2)}{n(n+1)}k_s j_n(k_s \|x\|)\mathrm{Grad}\, Y_n^m(\hat{x}).
\end{split}
\end{equation}
Thus, by the condition $\nu\wedge \big( \nabla\wedge \check{U}_s \big)=0$ on $\partial B_R$, one has
\[
\check{b}_n^m=\frac{\alpha_s}{k_s}\frac{A_n^m(2)}{n(n+1)} \frac{j_n(k_sR)}{h_n(k_sR)}\quad \text{and} \quad \check{c}_n^m=-\alpha_s\frac{A_n^m(3)}{n(n+1)} \frac{\check{j}_n(k_sR)}{\check{h}_n(k_sR)}.
\]
Moreover, it is observed from \eqref{eq:eq3161} and \eqref{eq:eq_checkU} along with the relation $\nu\cdot \check{U}_p=0$ in \eqref{eq:decompps232} on $\partial B_R$ that
\[
\check{a}_n^m=A_n^m(1) \frac{\alpha_p}{k_p} \frac{j'_n(k_pR)}{h'_n(k_pR)}.
\]

So far, we have obtained the total wave field $\check{U}=\check{U}_p+\check{U}_s$ in \eqref{eq:eq_checkU} which satisfies the decoupling relation \eqref{eq:decompps232} on $\partial B_R$. Next, we assert that
\begin{equation}\label{eq:assertion}
\check{U}\neq U(\cdot;[B_R, \mathrm{III}]).
\end{equation}
Noting that $\nu\cdot U=0$ on $\partial B_R$, we can first see from Lemma~\ref{lem:2} that
\[
\nu\wedge \mathscr{D}_{\nu}U=\nu\wedge\nabla(\nu\cdot U)-\frac{1}{\|x\|}\nu\wedge U=-\frac{1}{\|x\|}\nu\wedge U.
\]
Thus, using \eqref{eq:pp37}, one observes that
\[
\nu\wedge TU=-2\mu\, \frac{1}{\|x\|}\, \nu\wedge U-\mu\, \nu\wedge\big(\nu\wedge (\nabla\wedge U) \big)\quad\mbox{on $\partial B_R$}.
\]
Define a new operator $\widetilde{T}$ on $\partial B_R$ by
\[
\widetilde{T}U:=-2\mu \frac{1}{\|x\|} U-\mu \nu\wedge \big(\nabla\wedge U\big).
\]
It is easily seen that
\[
\nu\wedge \widetilde{T}U=\nu\wedge TU=0 \quad\mbox{on $\partial B_R$}.
\]
Next we establish the representation of $\nu\wedge \widetilde{T}\check{U}$. Notice that $\nu\wedge \big(\nabla\wedge\check{U}\big)=0$ on $\partial B_R$. Thus we have
$$
\nu\wedge \widetilde{T}\check{U}=-2\mu \frac{1}{\|x\|}\nu\wedge \check{U}.
$$
Using \eqref{eq:eq3161}, \eqref{eq:eq3162} and \eqref{eq:eq_checkU}, one can derive on $\partial B_R$ with straightforward computations that
\begin{equation}\label{eq:eq_TUsch}
\begin{split}
  \nu\wedge \widetilde{T}\check{U}_s= & -2i \mu \alpha_s\frac{1}{R} \sum_{n=1}^{\infty}\frac{1}{n(n+1)}\sum_{m=-n}^{n} \frac{A_n^m(2)}{t_s^2 h_n(t_s)}\hat{x}\wedge\mathrm{Grad}\, Y_n^m(\hat{x}) \\
    & -2i \mu \alpha_s\frac{1}{R} \sum_{n=1}^{\infty}\frac{1}{n(n+1)}\sum_{m=-n}^{n}\frac{A_n^m(3)}{t_s \check{h}_n(t_s)}\mathrm{Grad}\, Y_n^m(\hat{x}),
\end{split}
\end{equation}
and that
\begin{equation}\label{eq:eq_TUpch}
  \nu\wedge \widetilde{T}\check{U}_p=2i \mu \alpha_p\frac{1}{R} \sum_{n=0}^{\infty}\sum_{m=-n}^{n} \frac{A_n^m(1)}{t_p^3 h'_n(t_p)}\hat{x}\wedge\mathrm{Grad}\,Y_n^m(\hat{x}).
\end{equation}
It is obvious that $\nu\wedge \widetilde{T}\check{U}$ is not identically zero on $\partial B_R$, which verifies our assertion \eqref{eq:assertion}.

By summarizing our earlier discussion, we have that

\begin{prop}
There hold
\[
U_p^{sc}(\cdot;[B_R, \mathrm{III}])=\check{U}_p^{sc}+\sum_{n=1}^{\infty}\sum_{m=-n}^{n} \beta_n^m \nabla u_n^m[h;k_p],
\]
and
\[
U_s^{sc}(\cdot;[B_R, \mathrm{III}])=\check{U}_s^{sc}+\sum_{n=1}^{\infty}\sum_{m=-n}^{n} \big\{\gamma_n^m \nabla\wedge M_n^m[h;k_s]+\zeta_n^m M_n^m[h;k_s]\big\},
\]
where $\widetilde{U}^{sc}_p$ and $\widetilde{U}^{sc}_s$ are given in \eqref{eq:eq_checkU} such that $\check{U}_p:=U_p^{in}+\check{U}_p^{sc}$ and $\check{U}_s:=U_s^{in}+\check{U}_s^{sc}$ satisfying the decoupling condition \eqref{eq:decompps232} on $\partial B_R$, and the constants $\beta_n^m$, $\gamma_n^m$ and $\zeta_n^m$ are implicitly given by
\begin{equation}\label{eq:eq327}
\left\{
\begin{array}{l}
t_p h'_n(t_p)\beta_n^m-n(n+1)h_n(t_s)\gamma_n^m =0,\\
\displaystyle{
2h_n(t_p)\beta_n^m+\left(t_s^2 h_n(t_s)+2\check{h}_n(t_s)\right)\gamma_n^m=\frac{2i}{R}\left(\frac{\alpha_p}{k^2_p} \frac{A_n^m(1)}{t_p h'_n(t_p)} - \frac{\alpha_s}{k^2_s} \frac{A_n^m(2)}{h_n(t_s)} \right),}
\end{array}
\right.
\end{equation}
and
\[
\zeta_n^m=-2i\frac{A_n^m(3)}{n(n+1)}\frac{\alpha_s}{t_s\widetilde{h}_n(t_s)\check{h}_n(t_s)},
\]
with $t_p=k_p R$ and $t_s=k_s R$.

Furthermore, the total field $U(\cdot;[B_R, \mathrm{III}])$ satisfies the decoupling property \eqref{eq:decompps232} on $\partial B_R$ if and only if $\beta_n^m=\gamma_n^m=\eta_n^m=0$ for all $m=-n,\ldots,n$, and $n\in\mathbb{N}$. This cannot be true unless $\alpha_s=\alpha_p=0$.
\end{prop}

\begin{proof}
%\[
%\begin{split}
%  \nu\wedge\left(\nu\wedge\text{curl} U_s^{in}\right) &= -\alpha_s\sum_{n=1}^{\infty}\sum_{m=-n}^{n} \frac{A_n^m(3)}{n(n+1)}\frac{\check{j}_n(k_s|x|)}{|x|}\text{Grad}Y_n^m(\hat{x}) \\
%    & -\alpha_s\sum_{n=1}^{\infty}\sum_{m=-n}^{n} \frac{A_n^m(2)}{n(n+1)}k_s j_n(k_s|x|)\hat{x}\wedge\text{Grad}Y_n^m(\hat{x}).
%\end{split}
%\]
First, one has on $\partial B_R$ that
\[
\begin{split}
  \nu\cdot U_p & =\nu\cdot \check{U}_p +\nu\cdot\sum_{n=1}^{\infty}\sum_{m=-n}^{n} \beta_n^m \nabla u_n^m[h;k_p]\\
    & =\sum_{n=1}^{\infty}\sum_{m=-n}^{n}\beta_n^m k_p h'_n(k_p R)Y_n^m(\hat{x}),
\end{split}
\]
and
\[
\begin{split}
\nu\cdot U_s&= \nu\cdot \check{U}_s +\nu\cdot\sum_{n=1}^{\infty}\sum_{m=-n}^{n} \gamma_n^m \nabla\wedge M_n^m[h;k_s]\\
    & =\sum_{n=1}^{\infty}n(n+1)\sum_{m=-n}^{n}\gamma_n^m \frac{h_n(k_s R)}{|R|}Y_n^m(\hat{x}).
\end{split}
\]
Hence, the first equality in \eqref{eq:eq327} can be obtained from $\nu\cdot U=0$ in the third kind boundary condition \eqref{eq:thirdkind} on $\partial B_R$.

Next, it is observed that
\[
\begin{split}
  \nu\wedge \widetilde{T}\left(U_p-\check{U}_p\right)& =-2\mu \frac{1}{R}\nu\wedge\sum_{n=1}^{\infty}\sum_{m=-n}^{n} \beta_n^m \nabla u_n^m[h;k_p] \\
    & =-\mu\nu\wedge\sum_{n=1}^{\infty}\sum_{m=-n}^{n}2 \beta_n^m \frac{h_n(k_pR)}{R^2}\hat{x}\wedge\mathrm{Grad}\, Y_n^m(\hat{x}),
\end{split}
\]
and
\begin{align*}
    & \quad \nu\wedge \widetilde{T}\left(U_s-\check{U}_s\right) \\
  = & -\mu \nu\wedge \left(\frac{2}{\|x\|}+\nu\wedge\nabla\wedge \right) \sum_{n=1}^{\infty}\sum_{m=-n}^{n} \big\{\gamma_n^m \nabla\wedge M_n^m[h;k_s]+\zeta_n^m M_n^m[h;k_s]\big \}\\
  =&-\mu \sum_{n=1}^{\infty}\sum_{m=-n}^{n}\gamma_n^m \left(k_s^2 h_n(k_sR)+2\frac{\check{h}(k_sR)}{R^2}\right) \hat{x}\wedge\mathrm{Grad}\, Y_n^m(\hat{x})\\
 & \quad-\mu \sum_{n=1}^{\infty}\sum_{m=-n}^{n}\zeta_n^m \frac{\widetilde{h}_n(k_sR)}{R}\mathrm{Grad}\, Y_n^m(\hat{x}).
\end{align*}
By using the above facts, the rest of the proof can be concluded by using the boundary condition $\nu\wedge \widetilde{T}U=0$ in \eqref{eq:thirdkind} on $\partial B_R$, along with the use of \eqref{eq:eq_TUsch} and \eqref{eq:eq_TUpch}.

\end{proof}

\section{Inverse elastic scattering problems}~\label{sect:4}

In this part of the paper, we consider the inverse elastic scattering problems given in \eqref{eq:ip} with the elastic scatterer being $[D;\mathrm{III}]$ or $[D;\mathrm{IV}]$. We call $[D;\mathrm{III}]$ and $[D;\mathrm{IV}]$ admissible if they satisfy the geometric conditions \eqref{eq:gcmc} and \eqref{eq:gc32}, respectively. The crucial ingredient is the following decoupling theorem which can be readily summarized from our study in Section~\ref{sect:2}.

\begin{thm}\label{thm:3}
Let $D$ be an admissible scatterer and let $U$ be the solution to the Lam\'e system \eqref{eq:lames}--\eqref{eq:radiation}. Define
\begin{equation}\label{eq:ddcc1}
v_p:=-\nabla\cdot U\quad\mbox{and}\quad E_s:=\nabla\wedge U.
\end{equation}
Then there hold
\begin{equation}\label{eq:ddcc2}
U=U_s+U_p,\quad U_s=\frac{1}{k_s^2}\nabla\wedge E_s\quad\mbox{and}\quad U_p=\frac{1}{k_p^2}\nabla v_p.
\end{equation}
Moreover, set
\begin{equation}\label{eq:ddccin}
v_p^{in}:=-\nabla\cdot U^{in}\quad\mbox{and}\quad E_s^{in}:=\nabla\wedge U^{in}.	
\end{equation}
Then $v_p$ satisfies an acoustic scattering system (cf. \eqref{eq:dsp3} or \eqref{eq:dsp332}), and $E_s$ satisfies an electromagnetic scattering system (cf. \eqref{eq:dsp61} or \eqref{eq:dsp612}). The correspondences are given as follows,
\begin{equation}\label{eq:ddcccor}
\begin{split}
& v_p(v_p^{in}; [D, \mathrm{C(K)}])=-\nabla\cdot U(U^{in}; [D, \mathrm{K}]),\\
& E_s(E_s^{in}; [D, \mathrm{B(K)}])=\nabla\wedge U(U^{in}; [D, \mathrm{K}]),
\end{split}
\end{equation}
where the notations $\mathrm{C(K)}$ and $\mathrm{B(K)}$ are defined by
\begin{equation}\label{eq:ddccno}
\begin{split}
& \mathrm{C(K)}=\mathrm{SH}\ \text{and}\ \mathrm{B(K)}=\mathrm{PEC}, \quad \,\text{if} \ \ \mathrm{K}=\mathrm{III};\\
& \mathrm{C(K)}=\mathrm{SS}\ \text{and}\ \,\mathrm{B(K)}=\mathrm{PMC}, \quad \text{if} \ \ \mathrm{K}=\mathrm{IV}.
\end{split}
\end{equation}
%\begin{enumerate}
%	\item $v_p(v_p^{in}; [D, \mathrm{SH}])=-\nabla\cdot U(U^{in}; [D, \mathrm{III}])$; %with $v_p^{in}$ given in \eqref{eq:dsp2}.
%	\item $E_s(E_s^{in}; [D, \mathrm{PEC}])=\nabla\wedge U(U^{in}; [D, \mathrm{III}])$; % with $E_s^{in}$ given in \eqref{eq:dsp5half}.
%	\item $v_p(v_p^{in}; [D, \mathrm{SS}])=-\nabla\cdot U(U^{in}; [D, \mathrm{IV}])$; % with $v_p^{in}$ given in \eqref{eq:dsp2}.
%	\item $E_s(E_s^{in}; [D, \mathrm{PMC}])=\nabla\wedge U(U^{in}; [D, \mathrm{IV}])$; % with $E_s^{in}$ given in \eqref{eq:dsp5half}.
%\end{enumerate}
\end{thm}

The next proposition establishes the correspondences between the far-fields patterns of the decomposed wave fields in Theorem~\ref{thm:3}, which can be obtained by direct calculations.

\begin{prop}\label{prop:ddcc1}
Let $U, U_s, U_p$ and $v_p, E_s$ be the scattering wave fields given in Theorem~\ref{thm:3}. Let $U_t^\infty, U_s^\infty$  and $U_p^\infty$ be given in \eqref{eq:farfieldt}, namely,
\begin{equation}\label{eq:farfieldt3}
U_t^\infty(\hat x):=U_p^\infty(\hat x)+U_s^\infty(\hat x).
\end{equation}
Let $v_p^\infty$ and $E_s^\infty$ be the far-field patterns, respectively, of $v_p$ and $E_s$ that are read off from the following asymptotic expansions as $\|x\|\rightarrow+\infty$,
\begin{equation}\label{eq:ddcc3}
v_p(x)=v_p^{in}(x)+\frac{e^{ik_p\|x\|}}{\|x\|} v_p^\infty(\hat x)+\mathcal{O}\left(\frac{1}{\|x\|^{2}}\right),
\end{equation}
and
\begin{equation}\label{eq:ddcc3}
E_s(x)=E_s^{in}(x)+\frac{e^{ik_s\|x\|}}{\|x\|} E_s^\infty(\hat x)+\mathcal{O}\left(\frac{1}{\|x\|^{2}}\right).
\end{equation}
Then there hold
\begin{equation}\label{eq:ddcc4}
U_p^\infty(\hat x)=\frac{i}{k_p} v_p^\infty(\hat x)\hat x,\qquad U_s^\infty(\hat x)=\frac{i}{k_s}\hat x\wedge E_s^\infty(\hat x),
\end{equation}
and
\begin{equation}\label{eq:ddcc5}
v_p^\infty(\hat x)=-ik_p\hat x\cdot U_p^\infty(\hat x),\qquad E_s^\infty(\hat x)=ik_s\hat x\wedge U_s^\infty(\hat x).
\end{equation}
\end{prop}

Clearly, by using the correspondences given in Theorem~\ref{thm:3} and Proposition~\ref{prop:ddcc1}, the relevant results for the inverse acoustic and electromagnetic scattering problems can be extended to the inverse elastic scattering problems with admissible scatterers of the third and fourth kinds. We are particularly interested in the inverse problem of determining $D$ in \eqref{eq:ip} by knowledge of a single far-field measurement; that is, $U_\tau^\infty(\hat x; U^{in})$, $\tau=t, p$ or $s$, is collected corresponding to a fixed incident plane wave $U^{in}$. There is a widespread belief that the uniqueness can be established for such an inverse problem. However, it still remains to be challengingly open. As discussed in Section~\ref{sect:1}, in a recent work \cite{ElY10}, the authors established two uniqueness results in determining polyhedral scatterer of the third and fourth kinds. By using the decoupling results in Theorem~\ref{thm:3} and Proposition~\ref{prop:ddcc1}, it can be shown that much more general uniqueness results hold for the inverse elastic scattering problems with polyhedral scatterers. Moreover, we can step much further in driving optimal stability estimates for the inverse elastic problem through the extension of the recent stability results in \cite{LMRX} for the inverse acoustic scattering problems. It is emphasized that depending on the presence of screen-components or not, we shall make use of a single or two or three far-field measurements in the subsequent uniqueness study, and these are known to be the minimum number of measurements required. 

\subsection{Uniqueness results}

\begin{defn}\label{def:p}
A bounded open subset of a two-dimensional plane is called a \emph{cell}. A set $D$ in $\mathbb{R}^3$ is said to be a \emph{polyhedral scatterer} if $\overline{D}$ is compact, $\mathbb{R}^3\backslash\overline{D}$ is connected and $\partial \overline{D}$ consists of finitely many closures of cells. A polyhedral scatterer $D$ is said to be a \emph{polyhedral obstacle} if $D$ is open.
\end{defn}

In what follows, we let $\mathscr{P}_s$ and $\mathscr{P}_o$ denote the classes of polyhedral scatterers and obstacles, respectively. Clearly, $\mathscr{P}_o\Subset\mathscr{P}_s$. We note for clarity that for a scatterer in $\mathscr{P}_s$, it may contain a component which is solely a cell (it is referred to as a screen-component), whereas for an obstacle in $\mathscr{P}_o$, all of its components must be solid. It is easily seen that any $[D,\mathrm{III}]$ or $[D, \mathrm{IV}]$ contained in $\mathscr{A}_s$ is admissible which satisfies the geometric conditions \eqref{eq:gcmc} and \eqref{eq:gc32}. If $D\in\mathscr{P}_o$, then the corresponding forward elastic scattering problem is well-understood, completely covered in our earlier discussion. However, if $D\in\mathscr{P}_s$, we cannot find a convenient reference for the well-posedness of the corresponding forward elastic scattering problem. In the sequel, when a scatterer $D\in\mathscr{P}_s$ is considered, we always assume the unique existence of an $H_{loc}^1(\mathbb{R}^3\backslash\overline{D})$ solution to the forward scattering problem, and focus on the study of the corresponding inverse problems, which is the main theme of the current section. Moreover, in order to ease the discussion, we present the uniqueness results for general polyhedral scatterers, whereas we only discuss the stability results for polyhedral obstacles.

\begin{thm}\label{thm:uniquenesso}
Let $D$ and $D'$ be two polyhedral obstacles in $\mathscr{P}_o$ such that
\begin{equation}\label{eq:uo1}
U_t^\infty(\hat x; U^{in}, [D, \mathrm{K}])=U_t^\infty(\hat x; U^{in}, [D', \mathrm{K}'])\quad\forall \hat x\in\mathbb{S}^2,
\end{equation}
where $U^{in}$ is given in \eqref{eq:planewave} with either $\alpha_p\neq 0$ or $\alpha_s\neq 0$, and $\mathrm{K}, \mathrm{K}'=\mathrm{III}$ or $\mathrm{IV}$ denote the respective kinds of the obstacles. Then one has
\begin{equation}\label{eq:uod1}
[D, \mathrm{K}]=[D', \mathrm{K}'].	
\end{equation}
That is, both the obstacle and its physical type can be uniquely determined by a single total far-field measurement. Furthermore, if $\alpha_p\neq 0$, $U_t^\infty$ in \eqref{eq:uo1} can be replaced by $U_p^\infty$; and if $\alpha_s\neq 0$, $U_t^\infty$ in \eqref{eq:uo1} can be replaced by $U_s^\infty$, then the same conclusion holds true.
\end{thm}

\begin{proof}
First, we assume that $\alpha_p\neq 0$. By the correspondences in Theorem~\ref{thm:3} and Proposition~\ref{prop:ddcc1}, we have from \eqref{eq:uo1}
that
\begin{equation}\label{eq:uo2}
v_p^\infty(\hat x; v_p^{in}, [D, \mathrm{C(K)}])=v_p^\infty(\hat x; v_p^{in}, [D', \mathrm{C(K')}])\quad\forall\hat x\in\mathbb{S}^2,
\end{equation}
for $\mathrm{K}, \mathrm{K}'=\mathrm{III}$ or $\mathrm{IV}$, with $\mathrm{C(K)}$ given by \eqref{eq:ddccin}. Then by the corresponding uniqueness result in \cite{Liu3} for inverse acoustic scattering problems, one readily deduces \eqref{eq:uod1}. Next, in the case $\alpha_s\neq 0$, we have from \eqref{eq:uo1} that
\begin{equation}\label{eq:uo3}
E_s^\infty(\hat x; E_s^{in}, [D, \mathrm{B(K)}])=E_s^\infty(\hat x; E_s^{in}, [D',\mathrm{B(K')}])\quad\forall\hat x\in\mathbb{S}^2,
\end{equation}
where $\mathrm{K}, \mathrm{K}'=\mathrm{III}$ or $\mathrm{IV}$, and the notation $\mathrm{B(K)}$ is defined in \eqref{eq:ddccin}. Now, the uniqueness can be readily deduced from the corresponding results for inverse electromagnetic scattering problems in \cite{Liu4}. The other cases with $U_t^\infty$ replaced by $U_p^\infty$ or $U_s^\infty$ can be treated in a completely similar manner.

The proof is complete.
\end{proof}

\begin{thm}\label{thm:uniquenesss}
Let $[D,\mathrm{K}]$ and $[D',\mathrm{K}]$ be two polyhedral scatterers in $\mathscr{P}_s$ where $\mathrm{K}=\mathrm{III}$ or $\mathrm{IV}$ represents the type of the scatterers.
Then one has
\begin{equation}\label{eq:usd1}
	D=D',
\end{equation}
provided
\begin{equation}\label{eq:us1}
	U_t^\infty(\hat x; U^{in}, [D, \mathrm{K}])=U_t^\infty(\hat x; U^{in}, [D', \mathrm{K}])\quad\forall \hat x\in\mathbb{S}^2,
\end{equation}
hold(s) for one of the following cases:
\begin{enumerate}
	\item[(i)] $\mathrm{K}=\mathrm{IV}$ and $\alpha_p\neq 0$, along with one single incident field $U^{in}=U^{in}(\cdot; d, d^{\perp},\alpha_p,\alpha_s,\omega)$ given by \eqref{eq:planewave};
	\item[(ii)] $\mathrm{K}=\mathrm{IV}$ and $\alpha_s\neq 0$, with two incident fields given by $U^{in}=U^{in}(\cdot;d,d_j^{\perp},\alpha_p,\alpha_s,\omega)$, $j=1,2$, such that $d$ and $d_j^{\perp}$, $j=1,2$, are linearly independent;
	\item[(iii)] $\mathrm{K}=\mathrm{III}$ and $\alpha_p\neq 0$, with the incident  $U^{in}=U^{in}(\cdot;d_j,d^{\perp},\alpha_p,\alpha_s,\omega)$, $j=1,2,3$, such that $d_1$, $d_2$ and $d_3$ are linearly independent;
	\item[(iv)] $\mathrm{K}=\mathrm{III}$ and $\alpha_s\neq 0$, with $U^{in}=U^{in}(\cdot;d,d_j^{\bot},\alpha_p,\alpha_s,\omega)$, $j=1,2$, such that $d_1^{\bot}$ and $d_2^{\bot}$ are linearly independent.
\end{enumerate}

Furthermore, if $\alpha_p\neq 0$ (namely, in the cases (i) and (iii)), $U_t^\infty$ in \eqref{eq:us1} can be replaced by $U_p^\infty$; and if $\alpha_s\neq 0$ (in the cases (ii) and (iv)), $U_t^\infty$ in \eqref{eq:uo1} can be replaced by $U_s^\infty$, then the same conclusion holds true.
\end{thm}

\begin{proof}
Similar to the proof of Theorem~\ref{thm:uniquenesso}, the primary strategy here is the application of Theorem~\ref{thm:3} and Proposition~\ref{prop:ddcc1}, along with the use of the relevant uniqueness results for inverse acoustic and electromagnetic scattering problems. 

If $\mathrm{K}=\mathrm{IV}$ and $\alpha_p\neq 0$, similar to \eqref{eq:uo2}, we first obtain that
\begin{equation}\label{eq:us2}
v_p^\infty(\hat x; v_p^{in}, [D, \mathrm{SS}])=v_p^\infty(\hat x; v_p^{in}, [D', \mathrm{SS}])\quad\forall\hat x\in\mathbb{S}^2,
\end{equation}
with $v_p^{in}$ given by \eqref{eq:ddccin}.
By the corresponding uniqueness result in inverse acoustic scattering problems for sound-soft polyhedral scatters (cf. \cite{AR05,Liu1}), one readily has \eqref{eq:usd1}. In the case that $\mathrm{K}=\mathrm{III}$ and $\alpha_p\neq 0$, we obtain from \eqref{eq:ddcccor} that
\begin{equation}\label{eq:us21}
v_p^\infty(\hat x; v_{p,j}^{in}, [D, \mathrm{SH}])=v_p^\infty(\hat x; v_{p,j}^{in}, [D', \mathrm{SH}])\quad\forall\hat x\in\mathbb{S}^2,\quad j=1,2,3,
\end{equation}
where
\begin{equation}
v_{p,j}^{in}:=-\nabla\cdot U^{in}=-(i\alpha_pk_p)e^{ik_px\cdot d_j}, \quad j=1,2,3,
\end{equation}
with three linearly independent directions $d_1$, $d_2$ and $d_3$. Hence, \eqref{eq:usd1} can be shown by the uniqueness results for inverse sound-hard problems (c.f. \cite{Liu1}).

If $\mathrm{K}=\mathrm{IV}$ with $\alpha_s\neq 0$, we have from Theorem~\ref{thm:3} and Proposition~\ref{prop:ddcc1} that
\begin{equation}\label{eq:us3}
E_s^\infty(\hat x; E_s^{in}, [D, \mathrm{PMC}])=E_s^\infty(\hat x; E_s^{in}, [D', \mathrm{PMC}])\quad\forall\hat x\in\mathbb{S}^2,
\end{equation}
with
\begin{equation}\label{eq:us4}
E_s^{in}(x)=\nabla\wedge U^{in}(x;d,d_j^{\perp},\alpha_p,\alpha_s,\omega)=i\alpha_s k_s\, ( d\wedge d_j^{\perp}) e^{ik_sd\cdot x},\quad j=1,2.
\end{equation}
Define
\begin{equation}\label{eq:us5}
\widetilde{E}_s:=-\frac{i}{k_s} \nabla\wedge E_s(\cdot; E_s^{in}, [D, \mathrm{PMC}]),\quad \widetilde{E}'_s:=-\frac{i}{k_s} \nabla\wedge E_s(\cdot; E_s^{in}, [D', \mathrm{PMC}]).
\end{equation}
Similar to \eqref{eq:dsp51}-\eqref{eq:dsp5half21}, one has that $(\widetilde{E}_s,-E_s(\cdot; E_s^{in}, [D, \mathrm{PMC}]))$ solves the Maxwell system \eqref{eq:dsp61} with $\widetilde{E}_s^{in}$ given in \eqref{eq:dsp52}, and the PMC boundary condition in \eqref{eq:dsp61} on $\partial D$ replaced by the PEC one \eqref{eq:dsp5half21}; and $(\widetilde{E}'_s,-E_s(\cdot; E_s^{in}, [D',$ $\mathrm{PMC}]))$ solves the same problem except for the scatterer $D$ replaced by $D'$. That is,
\begin{equation}\label{eq:us6}
\widetilde{E}_s=E_s(\cdot; \widetilde{E}_s^{in}, [D, \mathrm{PEC}]),\quad \widetilde{E}'_s= E_s(\cdot; \widetilde{E}_s^{in}, [D', \mathrm{PEC}]),
\end{equation}
with two incident plane waves
\begin{equation}\label{eq:us7}
\widetilde{E}_s^{in}=- i\alpha_s k_s\, d_j^{\perp} e^{ik_sd\cdot x},\quad j=1,2.
\end{equation}
Therefore, from \eqref{eq:us3}, \eqref{eq:us5} and \eqref{eq:us6}, one observes that
\begin{equation}
E_s^\infty(\hat x; \widetilde{E}_s^{in}, [D, \mathrm{PEC}])=E_s^\infty(\hat x; \widetilde{E}_s^{in}, [D', \mathrm{PEC}])\quad\forall\hat x\in\mathbb{S}^2,
\end{equation}
for $\widetilde{E}_s^{in}$ given in \eqref{eq:us7} with two linearly independent vectors $d_1^{\perp}$ and $d_2^{\perp}$.
Hence, \eqref{eq:usd1} can be obtained by the uniqueness results for inverse PEC problems in \cite{Liu2}.

In the last case for $\mathrm{K}=\mathrm{III}$ and $\alpha_s\neq 0$, similar to \eqref{eq:us3}, one observes that
\begin{equation}\label{eq:us8}
E_s^\infty(\hat x; E_{s,j}^{in}, [D, \mathrm{PEC}])=E_s^\infty(\hat x; E_{s,j}^{in}, [D', \mathrm{PEC}])\quad\forall\hat x\in\mathbb{S}^2,\quad j=1,2,
\end{equation}
with
\begin{equation}\label{eq:us9}
E_{s,j}^{in}(x)=i\alpha_s k_s\, ( d\wedge d_j^{\perp}) e^{ik_sd\cdot x},\quad j=1,2.
\end{equation}
It is easy to see that $d\wedge d_1$ and $d\wedge d_2$ are linearly independent, since $d_1$ and $d_2$ are linearly independent and both are perpendicular to $d$. Now, \eqref{eq:usd1} can be deduced in the same way as that for the case (iii).

Finally, for the other cases with $U_t^\infty$ replaced by $U_p^\infty$ or $U_s^\infty$, one can show the corresponding uniqueness results by using completely similar arguments as above. The proof is complete. 
\end{proof}

We would like to point out that, the two uniqueness results in \cite{ElY10} in determining an elastic polyhedral scatterer are, respectively, case (i) and part of case (iv) in Theorem~\ref{thm:uniquenesso}. 
%
%the uniqueness results in inverse elastic scattering acquired in \cite{ElY10} are totally covered in Theorem~\ref{thm:uniquenesso}. It is obtained by J. Elschner and M. Yamamoto in \cite{ElY10} that, an elastic polyhedral scatterer  $D=[D,\mathrm{IV}]$ can be uniquely determined by a single far-field measurement generated by one plane incident pressure wave $U^{in}_p$; while for polyhedral $D=[D,\mathrm{III}]$, it can be detected by two far-field measurements produced by two incident shear waves, which are of the same impinging direction but different polarization vectors, i.e., $U^{in}_s(\cdot; d, d_j^{\perp},\alpha_p,\alpha_s,\omega)$, $j=1,2$. It is seen that the former conclusion is covered by the case (i) in Theorem~\ref{thm:uniquenesso}, while the latter on is part of the case (iv) stated in Theorem~\ref{thm:uniquenesso}.
A crucial ingredient in \cite{ElY10} to prove the corresponding uniqueness results therein is the so called ``reflection principle'' for solutions to the Lam\'e system \eqref{eq:lames}, which was developed from the corresponding reflection principles for the Helmholtz equation and the Maxwell equations, respectively (cf. \cite{AR05,Liu1,Liu2}).  
By using the decoupling results in Theorem~\ref{thm:3}, the next proposition gives a simple proof of the reflection principle for the Lam\'e system. It can be straightforwardly used to explain the reflection principles in \cite{ElY10} that were critical in the proofs of the relevant uniqueness results. Moreover, the reflection principle might find important applications in other context on elastic wave scattering. 

\begin{prop}
Let the admissible scatterer $D$, and the functions $U$, $E_s$ and $v_p$, be the same as in Theorem~\ref{thm:3}. Let $\Pi$ be a two-dimensional plane, and let $\Pi'$ be the affine projection of $\Pi$, namely, $\Pi'$ be the parallel plane of $\Pi$ passing through the origin.
Set
\begin{equation}\label{eq:refle1}
\widetilde{U}(x):=R_{\Pi'}U(R_{\Pi}x),\quad x\in R_{\Pi} D^c,
\end{equation}
where $D^c$ is given by $D^c:=\mathbb{R}^3\setminus \overline{D}$, and $R_{\Pi}$ represents the geometric reflection with respect to the plane $\Pi$.
Then one has
\begin{equation}\label{eq:refle2}
\widetilde{U}=\widetilde{U}_p+\widetilde{U}_s:=\frac{1}{k_p^2}\nabla\,\widetilde{v}_p+\frac{1}{k_s^2}\nabla\wedge\widetilde{E}_s,
\end{equation}
and
\begin{equation}\label{eq:refle3}
\widetilde{v}_p=-\nabla\cdot \widetilde{U}\quad\mbox{and}\quad \widetilde{E}_s:=\nabla\wedge \widetilde{U},
\end{equation}
with
\begin{equation}\label{eq:refle4}
\widetilde{v}_p(x):=v_p(R_{\Pi}(x)),\quad \widetilde{E}_s(x):=-R_{\Pi'}E_s(R_{\Pi}x),\quad x\in R_{\Pi} D^c.
\end{equation}
\end{prop}

\begin{proof}
It is verified with straightforward computations that the reflected $\widetilde{v}_p$ defined in \eqref{eq:refle4} is invariant under gradient, namely,
\begin{equation}\label{eq:refle5}
\nabla\,\widetilde{v}_p(x)=R_{\Pi'}\left(\nabla v_p\right)(R_{\Pi}x).
\end{equation} 
Similarly, $\widetilde{E}_s$ given by \eqref{eq:refle4} is invariant under the curl operator, that is,
\begin{equation}\label{eq:refle6}
\nabla\wedge\widetilde{E}_s(x)=R_{\Pi'}\left(\nabla\wedge E_s\right)(R_{\Pi}x).
\end{equation}
Thus, by applying \eqref{eq:ddcc2} one can derive for any $x\in R_{\Pi} D^c$ that
\begin{equation}\label{eq:refle7}
\begin{split}
\widetilde{U}(x) %=\widetilde{U}_s(x)+\widetilde{U}_p(x)
& =R_{\Pi'}U(R_{\Pi}x)\\
& =\frac{1}{k_s^2}R_{\Pi'}\left(\nabla\wedge E_s\right)(R_{\Pi}x)+\frac{1}{k_p^2}R_{\Pi'}\left(\nabla v_p\right)(R_{\Pi}x)\\
&=\frac{1}{k_s^2}\nabla\wedge\widetilde{E}_s(x)+\frac{1}{k_p^2}\nabla\,\widetilde{v}_p(x).
\end{split}
\end{equation}
Taking gradient and curl, respectively, on both sides of \eqref{eq:refle7}, one readily obtains \eqref{eq:refle3}.

The proof is complete. 
\end{proof}

\subsection{Stability results}

In this section, we present the stability estimates in determining polyhedral obstacles by a single far-field measurement, which quantifies the uniqueness results in Theorem~\ref{thm:uniquenesso}. 

First, we introduce some preliminary definitions and statements before presenting the main stability results.

Given $n\in\mathbb{N}$, we say that a domain $\Omega\subset\mathbb{R}^n$ is Lipschitz with constant $r$ and $L$ if the following assumption holds.
For any $x\in\partial \Omega$, we first change rigidly the local coordinates such that $x=0$. There exists a function $\varphi_x:\mathbb{R}^{n-1}\to\mathbb{R}$, which is Lipschitz with Lipschitz constant bounded by $L$, such that $\varphi(0)=0$, and
$$B_r(x)\cap \Omega = \{y=(y',\,y_n)\in B_r(x);y_n\in\mathbb{R},\, y'\in\mathbb{R}^{n-1},\,y_n<\varphi(y')\},$$
and consequently,
$$B_r(x)\cap \partial\Omega = \{y=(y',\,y_n)\in B_r(x);y_n\in\mathbb{R},\, y'\in\mathbb{R}^{n-1},\, y_n=\varphi(y')\}.$$
%
%Clearly, $\partial\Omega$ is a Lipschitz hypersurface, without boundary, with the same constants $r$ and $L$. Again, the classes of sets $\overline{\Omega}$ and $\partial\Omega$, where $\Omega$ is an open set contained in $B_R$ which is Lipschitz with constants $r$ and $L$, are compact with respect to the Hausdorff distance.

A scatterer $D\subset\mathbb{R}^3$ is said to be polyhedral with constants $h$ and $L$, if $\partial \overline{D}$ is composed of a finite union of the closures of cells,
\begin{equation}\label{eq:2_2_Dp=}
\partial \overline{D}=\bigcup_{j=1}^{m} \overline{\mathcal{C}_j},
\end{equation}
where $m\in\mathbb{N}$ is a certain finite (unknown) number, and each $\mathcal{C}_j$, $1\leq j\leq m$, is a Lipschitz domain in $\mathbb{R}^2$ with constants $h$ and $L$; and two different cells may intersect only at boundary points.

In this paper, the distance between two scatterers $D$ and $D'$ in $\mathbb{R}^3$ is measured by the Hausdorff distance $d_H$,
\begin{equation}\label{eq:Haus}
d_H(D,D'):=\max\left\{ \sup_{x\in D}\text{dist}(x, D'), \ \sup_{x\in D'} \text{dist}(x, D) \right\}.
\end{equation}

Given positive constants $r$, $L$, $R$ and $h$, let $\mathscr{P}_o^h=\mathscr{P}_o^h(r,L,R)$ denotes the set of obstacles $D\subset B_{R}$ such that, $D$ is a Lipschitz domain with constants $r$ and $L$, and that $D$ is polyhedral with constants $h$ and $L$. In the rest of this section, we shall aways let the positive constants $r$, $L$ and $R$, and the elements $\alpha_p$, $\alpha_s$, $d$, $d^{\perp}$ and $\omega$ mentioned in \eqref{eq:planewave}, and the wave numbers $k_p$ and $k_s$ be fixed; they are all referred to as the \textsl{a priori data}. We also keep the other positive number $h$ being fixed. It is noted that $h$ is not included into the a priori data.

\begin{thm}
Let $[D,\mathrm{K}]$ and $[D',\mathrm{K}]$ be two polyhedral obstacles in $\mathscr{P}_o^h=\mathscr{P}_o^h(r,L,R)$, where $\mathrm{K}=\mathrm{III}$ or $\mathrm{IV}$ represents the type of the scatterers. Suppose that
\begin{equation}\label{eq:stabi1}
\|U_t^\infty(\cdot; U^{in}, [D, \mathrm{K}])-U_t^\infty(\cdot; U^{in}, [D', \mathrm{K}])\|_{L^2(\mathbb{S}^2)}\leq \varepsilon,
\end{equation}
for a single incident wave $U^{in}$ given by \eqref{eq:planewave}. Then there exists a positive number $\widetilde{\varepsilon}(h;\mathrm{K})$, depending only on $h$, on the a priori data, and on the scatterer type $\mathrm{K}$, such that if $\varepsilon\leq\widetilde{\varepsilon}(h;\mathrm{K})$, then one has
\begin{equation}\label{eq:d_main2}
d_H(D, D')\leq C [\psi(s\varepsilon)]^{\alpha},
\end{equation}
where $s$ is a certain number depending only on the a priori data which shall be specified below; the function $\psi:(0,1/e)\rightarrow (0,1)$ is defined by
\[
\psi(t)=\exp\left(-\left(\log\left(-\log t\right)\right)^{1/2}\right);
\]
$C$ is a positive constant depending only on the a priori data and the scatterer type $\mathrm{K}$, but indpendent on $h$; and the positive constant $\alpha$ depends only on the a priori data and the scatterer type $\mathrm{K}$, on or not on $h$; with more specifications as follows,
\begin{enumerate}
	\item[(i)] if $\mathrm{K}=\mathrm{IV}$ and $\alpha_p\neq 0$, then $s=k_p$, the constant $\alpha$ in \eqref{eq:d_main2} can be chosen so that it depends only on the a priori data, not on $h$; moreover, $U_t^\infty$ in \eqref{eq:stabi1} can be replaced by $U_p^\infty$;
	\item[(ii)] if $\mathrm{K}=\mathrm{III}$ and $\alpha_p\neq 0$, then $s=k_p$, the number $\alpha$ in \eqref{eq:d_main2} does depend on $h$ in general; and $U_t^\infty$ in \eqref{eq:stabi1} can also be replaced by $U_p^\infty$;
%	{\color{red}
%	\item[(iii)] if $\mathrm{K}=\mathrm{III}$ or $\mathrm{IV}$, and $\alpha_s\neq 0$, then the $U_t^\infty$ in \eqref{eq:stabi1} can be replaced by $U_s^\infty$; in this case, the $\alpha$ in \eqref{eq:d_main2} would depend on $h$.
%	}
\end{enumerate}
\end{thm}

\begin{proof}
Let $U^{sc}$ be a radiating solution to the L\'ame system \eqref{eq:lames}, and let $U_p^\infty$, $U_s^\infty$ and $U_t^\infty$ be the corresponding far-field pattern given by \eqref{eq:farfield} and \eqref{eq:farfieldt}. First, by \eqref{eq:farfieldd} one observes that 
\begin{equation}\label{eq:stabi2}
\|U_t^\infty\|_{L^2(\mathbb{S}^2)}=\|U_p^\infty\|_{L^2(\mathbb{S}^2)}+\|U_s^\infty\|_{L^2(\mathbb{S}^2)}.
\end{equation}	
It is clearly that $U_t(\cdot; U^{in}, [D, \mathrm{K}])-U_t(\cdot; U^{in}, [D', \mathrm{K}])$ is a radiating solution to the L\'ame system \eqref{eq:lames} in $\mathbb{R}^3\setminus (\overline{D\cup D'})$. Thus, \eqref{eq:stabi1} along with \eqref{eq:stabi2} implies that
\begin{equation}\label{eq:stabi3}
\xi(\tau,\mathrm{K}):=\|U_\tau^\infty(\cdot; U^{in}, [D, \mathrm{K}])-U_\tau^\infty(\cdot; U^{in}, [D', \mathrm{K}])\|_{L^2(\mathbb{S}^2)}\leq \varepsilon,
\end{equation}
for $\tau=p,s$ and $t$.

Denote
\begin{equation}
v_p(\mathrm{K}):=-\nabla\cdot U(\cdot;U^{in}, [D, \mathrm{K}])\quad \text{and} \quad v'_p(\mathrm{K}):=-\nabla\cdot U(\cdot;U^{in}, [D', \mathrm{K}]).
\end{equation}
It is deduced by Theorem~\ref{thm:3} that
\begin{equation}
v_p(\mathrm{K})=v_p(v_p^{in}; [D, \mathrm{C(K)}])\quad \text{and} \quad v'_p(\mathrm{K})=v_p(v_p^{in}; [D', \mathrm{C(K)}]),
\end{equation}
with $v_p^{in}$ and $C(K)$ given by \eqref{eq:ddccin} and \eqref{eq:ddccno}, respectively. Therefore, one observes from Proposition~\ref{prop:ddcc1} that
\begin{equation}
\|v_p^\infty(\mathrm{K})-{v'_p}^\infty(\mathrm{K})\|_{L^2(\mathbb{S}^2)}\leq k_p\xi(p,\mathrm{K})\leq k_p\varepsilon,
\end{equation}
where $v_p^\infty(\mathrm{K})$ ${v'_p}^\infty(\mathrm{K})$ are defined similarly to $v_p^\infty$ in Proposition~\ref{prop:ddcc1}. Now, the stability results in case (i) can be obtained by the corresponding stability estimate of inverse acoustic scattering for polyhedral sound-soft obstacles in \cite{Ron2}; and in case (ii) can be concluded by the stability estimate for determining sound-hard obstacles in \cite{LMRX}.

The proof is complete. 

\end{proof}

\section{Concluding remarks}

In this paper, we consider time-harmonic elastic wave scattering governed by the Lam\'e system associated with a third or fourth kind scatterer. It is well know that the elastic wave field can be decomposed into the pressure (longitudinal) and shear (transversal) parts. Generally, the pressure and shear waves coexist, but propagating at different speeds. We derive two geometric conditions, related to the mean and Gaussian curvatures of the boundary surface of the underlying scatterer, which can ensure the completely decoupling of the pressure and shear waves. Definitely, this phenomenon is of significant physical interest. On the other hand, if the decoupling occurs, then mathematical reduction occurs where the elastic wave equation can be decoupled into two simpler Helmholtz and Maxwell systems. Using the decoupling results, we obtain comprehensive uniqueness results for inverse elastic scattering problem in determining polyhedral scatterers by the minimal number of far-field measurements. Furthermore, we quantify the uniqueness result in determining polyhedral obstacles to derive optimal stability estimate of logarithmic type by a single far-field measurement. This is the first stability result in the literature for inverse elastic obstacle scattering problem by a single far-field measurement. We believe that the decoupling results can find more important applications in other context. For example, in \cite{HuL}, the invisibility cloaking of elastic waves was considered. Since the pressure and shear waves propagate at different speeds, one may first decouple them and then cloak them separately. In doing so, more favorable cloaking devices can be achieved. We shall explore such an idea in a forthcoming article.

\section*{Acknowledgement}

The work was supported by the FRG grants from Hong Kong Baptist University, Hong Kong RGC General Research Funds, 12302415 and 405513, and the NSF grant of China, No. 11371115.

\end{document}